\documentclass[a4paper,10pt]{amsart}

\usepackage[top=2.6cm,bottom=3cm, outer=2cm, inner=2cm]{geometry}
\usepackage{amsfonts,amsmath,amssymb,bbm,amsthm}
\usepackage{enumitem}
\usepackage{stmaryrd}
\usepackage{color,pdflscape} 
\usepackage{hyperref} 

\begin{document}
\newcommand{\R}{\mathbb R}
\newcommand{\Z}{\mathbb Z}
\newcommand{\E}{\mathbb E}
\newcommand{\1}{\mathbbm 1}
\newcommand{\N }{\mathbb N}
\newcommand{\q }{\mathbb Q}
\newcommand{\p }{\mathbb P}
\newcommand{\rr }{\mathsf R}
\newcommand{\D }{\mathcal D}
\newcommand{\F }{\mathcal F}
\newcommand{\G }{\mathcal G}
\newcommand{\h }{\mathcal H}
\newcommand{\M }{\mathcal M}
\newcommand{\eps}{\varepsilon}
\newcommand{\argmax}{\operatornamewithlimits{argmax}}
\newcommand{\argmin}{\operatornamewithlimits{argmin}}
\newcommand{\esssup}{\operatornamewithlimits{esssup}}
\newcommand{\essinf}{\operatornamewithlimits{essinf}}
\newcommand{\sgn}{\operatorname{sgn}}

\newcommand{\nc}{\newcommand}
\nc{\bg}{\begin} \nc{\e}{\end} \nc{\bi}{\begin{itemize}} \nc{\ei}{\end{itemize}} \nc{\be}{\begin{enumerate}} \nc{\ee}{\end{enumerate}} \nc{\bc}{\begin{center}} 
\nc{\ec}{\end{center}} \nc{\fn}{\footnote} \nc{\bs}{\backslash} \nc{\ul}{\underline} \nc{\ol}{\overline} 
\nc{\np}{\newpage}  \nc{\fns}{\footnotesize} 
\nc{\scs}{\scriptsize} \nc{\RA}{\Rightarrow} \nc{\ra}{\rightarrow} \nc{\bfig}{\begin{figure}} \nc{\efig}{\end{figure}} \nc{\can}{\citeasnoun} 
\nc{\vp}{\vspace} \nc{\hp}{\hspace}\nc{\LRA}{\Leftrightarrow}\nc{\LA}{\Leftarrow}
\renewcommand{\tilde}{\widetilde}
\nc{\eq}{\end{equation}} 

\nc{\ch}{\chapter}
\nc{\s}{\section}
\nc{\subs}{\subsection}
\nc{\subss}{\subsubsection}

\newtheorem{thm}{Theorem}[section]
\newtheorem{cor}[thm]{Corollary}
\newtheorem{lem}[thm]{Lemma}

\theoremstyle{remark}
\newtheorem{rem}[thm]{Remark}

\theoremstyle{example}
\newtheorem{ex}[thm]{Example}

\theoremstyle{ass}
\newtheorem{ass}[thm]{Assumption}

\theoremstyle{definition}
\newtheorem{df}[thm]{Definition}

\newenvironment{rcases}{
  \left.\renewcommand*\lbrace.
  \begin{cases}}
{\end{cases}\right\rbrace}

\setlength\parindent{0pt}
\pagestyle{plain}
\setitemize{leftmargin=20pt} 
\setenumerate{leftmargin=20pt} 

\title{A diffusion approximation for limit order book models}


   \author{Ulrich Horst}
   \address{Humboldt-Universit\"at zu Berlin, Germany}
   \email{horst@math.hu-berlin.de}

   \author{D\"orte Kreher}
   \address{Humboldt-Universit\"at zu Berlin, Germany}
   \email{kreher@math.hu-berlin.de}

	\begin{abstract}
	
This paper derives a diffusion approximation for a sequence of discrete-time one-sided limit order book models with non-linear state dependent order arrival and cancellation dynamics. The discrete time sequences are specified in terms of an $\R_+$-valued best bid price process and an $L^2_{loc}$-valued volume process. It is shown that under suitable assumptions the sequence of interpolated discrete time models is relatively compact in a localized sense and that any limit point satisfies a certain infinite dimensional SDE. Under additional assumptions on the dependence structure we construct two classes of models, which fit in the general framework, such that the limiting SDE admits a unique solution and thus the discrete dynamics converge to a diffusion limit in a localized sense. 

	\end{abstract}

  \subjclass[2010]{60F17, 91G80}
   \keywords{Functional limit theorem, diffusion limit, scaling limit, convergence of stochastic differential equations, limit order book}

\thanks{This research was partially supported by CRC 649: Economic Risk. Moreover, part of this research was performed while the second author was visiting the Institute for Pure and Applied Mathematics (IPAM), which is supported by the National Science Foundation. A previous version of this paper was entitled ``A functional convergence theorem for interpolated Markov chains to an infinite dimensional diffusion with application to limit order books.''}

\maketitle

\renewcommand{\baselinestretch}{1.15}\normalsize
\setlength{\parskip}{5pt}

\s{Motivation and setup}

In modern financial markets almost all transactions are settled through limit oder books (LOBs). A LOB is a record of unexecuted orders awaiting execution. Stochastic analysis provides powerful tools for understanding the complex system of order aggregation and execution in limit order markets via the description of suitable scaling (``high-frequency'') limits. Scaling limits allow for a tractable description of the macroscopic LOB dynamics (prices and standing volumes) from the underlying microscopic dynamics (individual order arrivals and cancellations). In this paper we prove a novel functional convergence result for a class of Markov chains arising in microstructure  models of LOBs to an infinite dimensional diffusion.

Scaling limits for LOBs have recently attracted considerable attention in the probability and finnacial mathematics literature. Depending on the scaling assumptions either fluid limits (cf.~\cite{Gao,Guo,HK1,HP}) or diffusion limits (cf.~\cite{BHQ,Cont3,Rosenbaum}) can be derived. Fluid limits for the full order book were first studied in \cite{HP} and afterwards in \cite{HK1}, where it was shown that under certain assumptions on the scaling parameters the sequence of discrete-time LOB models converges in probability to the solution of a deterministic differential equation. 
Although there is some work on probabilistic LOB models that assumes an SPDE or measure-valued dynamics for the volume process (cf.~\cite{Hubalek,Marvin}), there is little work on the derivation of a measure valued diffusion limit starting from a microscopic (``event-by-event'') description of the limit order book. Two exceptions are the particular models considered in \cite{BHQ} and \cite{Lakner2}. The work \cite{BHQ} extends the models in \cite{HP} and \cite{HK1} by introducing additional noise terms in the pre-limit in which case the dynamics can then be approximated by an SPDE in the scaling limit. The papers \cite{BHQ, HK1, HP} rely on the same scaling assumptions. Our work is motivated by the question whether under {\it different} scaling assumptions the same event-by-event dynamics can be approximated by a diffusion process in the high frequency regime {\it without} adding additional noise terms in the pre-limit. 

\subsection{The LOB dynamics} 

The one-sided LOB models considered in this paper are specified by a sequence of discrete time $\R \times L^2(\R_+;\R)$-valued processes $\tilde{S}^{(n)}=\left(B^{(n)},v^{(n)}\right)$, where for each $n\in\N$,  the non-negative one dimensional process $B^{(n)}$ specifies the dynamics of the {\it best bid price}, and the $L^2(\R_+;\R)$-valued process $v^{(n)}$ specifies the dynamics of the bid-side {\it volume density function}. 

We fix some $T>0$ and introduce the scaling parameters $\Delta x^{(n)}, \Delta v^{(n)},$ and $\Delta t^{(n)}$. They denote the tick-size, the impact of an individual order on the state of the book, and the time between two consecutive order arrivals, respectively. We put $T_n:=\left\lfloor T /\Delta t^{(n)}\right\rfloor$, $x_j^{(n)}:=j\Delta x^{(n)}$ and $t_j^{(n)}:=j\Delta t^{(n)}\wedge T$ for all $j\in\N_0$ and $n\in\N$. For all $n\in\N$ and $x\in\R_+$ we define the interval $I^{(n)}(x)$ as
\[
	I^{(n)}(x):=\left[x_j^{(n)},x_{j+1}^{(n)}\right)\quad\text{for}\quad x_j^{(n)}\leq x<x_{j+1}^{(n)}.
\]

The initial best bid price is given by $B_0^{(n)}=b_n\Delta x^{(n)}$ for some $b_n\in\N$. The initial volume density function is given by a non-negative deterministic step function $v_0^{(n)}\in L^2(\R_+;\R)$  on the $\Delta x^{(n)}$-grid. Following the modelling framework of \cite{HK1} we assume that there are three events that change the state of the book: price increases (event $A$), price decreases (event $B$) and limit order placements, respectively cancellations (event $C$). In terms of the placement operator
\begin{equation}
	M_{k}^{(n)}(\cdot):=\1_{C}\left(\phi_k^{(n)}\right)\frac{\omega_k^{(n)}}{\Delta x^{(n)}} \1_{I^{(n)}\left(\pi_k^{(n)}\right)}(\cdot)
\end{equation}
the dynamics of the one-sided LOB models can then be described by the following point process:  for each $n\in\N$ and all $k=1,\dots,T_n$,  
\begin{equation}\label{Bv}
\begin{split}
	B_{k}^{(n)} &= B^{(n)}_{k-1}+\Delta x^{(n)}\left[\1_B\left(\phi_k^{(n)}\right)- \1_A\left(\phi_k^{(n)}\right)\right]\\
	v_{k}^{(n)} &= v_{k-1}^{(n)}+\Delta v^{(n)}M_{k}^{(n)}
\end{split}
\end{equation}
where the event indicator function $\phi_k^{(n)}$ is a random variable taking values in the set $\{A,B,C\}$, the $[-M,M]$-valued random variable $\omega_k^{(n)}$ specifies the size of a placement or cancellation $(M>0)$, and the non-negative random variable $\pi_k^{(n)}$ specifies the location of a placement or cancellation.  

\subsection{Preview of the main results}

In deriving a diffusion limit for the sequence of LOB models (\ref{Bv}), the first challenge is to define a suitable convergence concept. While for any $\pi\in\R_+$,
\[\left\Vert\1_{I^{(n)}\left(\pi\right)}\right\Vert_{L^2(\R_+)}=\left(\Delta x^{(n)}\right)^{1/2},\]
we have for any bounded $f\in L^2(\R_+)$,
\[\left\langle\1_{I^{(n)}\left(\pi\right)},f\right\rangle_{L^2(\R_+)}=\int_{I^{(n)}\left(\pi\right)}f(x)dx=\mathcal{O}\left(\Delta x^{(n)}\right).\]
Hence, it seems impossible to formulate a scaling assumption with $\Delta x^{(n)}\ra0,\ \Delta t^{(n)}\ra0,$ and $\Delta v^{(n)}\ra0$  that allows to prove convergence of the volume density functions to an $L^2(\R_+;\R)$-valued diffusion process. However, observe that for any $m,\pi>0$ we have
\[\left\Vert\Delta x^{(n)}\sum_{j=0}^{\lfloor \cdot/\Delta x^{(n)}\rfloor}\1_{I^{(n)}\left(\pi\right)}\left(x_j^{(n)}\right)\1_{[0,m]}(\cdot)\right\Vert_{L^2}=\Delta x^{(n)}
\left(m-\Delta x^{(n)}\left\lfloor\frac{\pi}{\Delta x^{(n)}}\right\rfloor\right)^{1/2}=
\mathcal{O}\left(\Delta x^{(n)}\right)\]
and for any bounded $f\in L^2$ also
\[\left\langle\Delta x^{(n)}\sum_{j=0}^{\lfloor \cdot/\Delta x^{(n)}\rfloor}\1_{I^{(n)}\left(\pi\right)}\left(x_j^{(n)}\right),f\1_{[0,m]}\right\rangle_{L^2} 
=\Delta x^{(n)}\int_{\Delta x^{(n)}\left\lfloor\frac{\pi}{\Delta x^{(n)}}\right\rfloor}^m f(x)dx=\mathcal{O}\left(\Delta x^{(n)}\right).\]
This suggests to study the convergence of the cumulated volume processes $V^{(n)}=\left(V^{(n)}_k\right)_{k\leq T_n}$ with 
\begin{equation}\label{cumv}
V^{(n)}_k(x):=\Delta x^{(n)}\sum_{j=0}^{\lfloor x/\Delta x^{(n)}\rfloor}v^{(n)}_k\left(x_j^{(n)}\right),\quad x\in\R_+,
\end{equation}
instead of analyzing directly the convergence of the volume density functions. To do this we will choose a localized convergence concept, since the functions $V^{(n)}$ are not square integrable on the whole line. 

Our main contribution is to establish a convergence concept and a convergence result for the sequence $S^{(n)} := \left(B^{(n)}, V^{(n)}\right)$, ${n \in \mathbb{N}}$. In particular, we state sufficient conditions that guarantee that (i) this sequence is relatively compact; (ii) any limit point solves an infinite dimensional SDE driven by a standard Brownian and a cylindrical Brownian motion; (iii) the limiting SDE has a unique solution. 

Having established a convergence concept, the second major challenge is that the dynamics of the process $S^{(n)}$, $n \in \N$, is not given in standard SDE form, due to the event-by-event dynamics, and that the system can only be controlled by specifying the conditional distribution of the random variables $\pi_k^{(n)}, \omega_k^{(n)},$ and $\phi_k^{(n)}$. Much of our work is, therefore, devoted to the identification of suitable integrands $G^{(n)}\left({S}^{(n)}(t)\right)$ and semimartingale random measures $Y^{(n)}$ such that $S^{(n)}(t)$ can be represented as
\begin{equation} \label{Sn}
	S^{(n)}(t) = S^{(n)}_0 +\int_0^tG^{(n)}\left({S}^{(n)}(u)\right)dY^{(n)}(u),\quad t\in[0,T]
\end{equation}
after continuous time-interpolation. Once the dynamics of the sequence $S^{(n)}$, $n \in \N$, has been brought into standard SDE form, it remains to study its convergence. The convergence of infinite dimensional stochastic integrals has been studied by several authors. Chao \cite{Cho} and Walsh \cite{Walsh} consider semimartingale random measures as distribution valued processes in some nuclear space. Kallianpur and Xiong \cite{Kallianpur} prove diffusion approximations of nuclear space-valued SDEs. Their approach requires a dependence structure that is incompatible with our spatial pointwise dynamics, and is hence not applicable to our modelling framework. Jakubowski \cite{Jakubowski} provides convergence results for Hilbert space valued semimartingales under a uniform tightness condition. Kurtz and Protter \cite{KP2} work with the same uniform tightness condition, but allow for a more general setting. Especially, they also study the convergence of solutions of stochastic differential equations in infinite dimension. The results are further extended by Ganguly \cite{Ganguly} to study the convergence of infinite dimensional stochastic differential equations when the approximating sequence of integrators is not uniformly tight anymore. 

Our proof relies on the results in \cite{KP2}. We first establish sufficient conditions that guarantee that the sequence $Y^{(n)}$, $n \in \N,$ converges to some $L^2(\R_+)^\#$-semimartingale $Y$. Subsequently we prove that the sequence $G^{(n)}$, $n \in \N,$ satisfies a compactness property and converges in a localised sense to some function $G$. Finally, we show that the sequence of stochastic differential equations in (\ref{Sn}) converges in law in a localised sense to a solution to an SDE of the form 
\begin{equation}
	{S}(t)=S_0+\int_0^tG\left({S}(u)\right)dY(u),\quad t\in[0,T].
\end{equation}
The challenge in proving the converges of the SDEs is the verification of the conditions in \cite{KP2} on the integrators and coefficient functions of the approximating sequence, and the fact that our convergence concept localises in space, not time. Finally, we give sufficient conditions for the uniqueness of solutions to the above SDE. For instance, we show that uniqueness holds if only the drift but not the volatility is state-dependent.

\subs{Structure of the paper}

The rest of the paper is structured as follows. In Section \ref{price} we state conditions on the dynamics of the price processes that guarantee the converge of their normalized fluctuations to a standard Brownian motion. In Section \ref{volume} we state conditions on the dynamics of the order arrivals and cancelations that guarantee convergence of the standardized fluctuations of the volume processes to a cylindrical Brownian motion. While the analysis of the price is quite standard, deriving similar results for the volumes is much more tedious. First we show in Subsections \ref{musigma} the convergence of the drift, volatility and correlation functions. Using an orthogonal decomposition of the covariance matrix we then establish in Subsection \ref{c} a representation of the volume process as a discrete stochastic differential equation driven by ``infinitely many discretised Brownian motions''. In Subsection \ref{gaussian} we prove the convergence in law of the ``infinitely many discretised Brownian motions'' to a cylindrical Brownian motion.  In Section \ref{integrals} we define the stochastic integrals and stochastic differential equations that describe the LOB dynamics and verify that the conditions from \cite{KP2} are satisfied. This allows us to derive our results on the characterisation of the limiting LOB dynamics as solutions to an infinite dimensional SDE in Section \ref{result}. We conclude with two specific examples in which the LOB dynamics converges weakly to the unique solution of an infinite dimensional SDE.

\subs{Notation} \label{setup}

For each $n\in\N$ we fix a probability space $\left(\Omega^{(n)},\F^{(n)},\p^{(n)}\right)$\footnote{For ease of notation we will simply write $\p$ and $\E$ in the following instead of $\p^{(n)}$ and $\E^{(n)}$, since it is clear from the context on which probability space we work.} with filtration 
\[\left\{\emptyset,\Omega^{(n)}\right\}=\F_0^{(n)}\subset \F_1^{(n)}\subset\dots\subset\F_k^{(n)}\subset\dots\subset\F^{(n)}_{T_n}\subset\F^{(n)}.\]
We assume that the random vector $\left(\phi_k^{(n)},\omega_k^{(n)},\pi_k^{(n)}\right)$ is $\F_k^{(n)}$-measurable for all $n\in\N$ and $k\leq T_n$. We define the Hilbert space
\[E:=\R \times L^2(\R_+;\R),\qquad \left\Vert (X_1,X_2)\right\Vert_E:=|X_1|+\left\Vert X_2\right\Vert_{L^2}\]
and its localized version
\[E_{loc}:=\R \times L^2_{loc}(\R_+;\R)\]
with
\[L^2_{loc}(\R_+):=\left\{f:\R_+\ra\R\ \left|\ \int_0^mf^2(x)dx<\infty\ \forall\ m\in\N\right.\right\}.\]
Moreover, we define for all $n\in\N$ the $E_{loc}$-valued stochastic process $S^{(n)}=\left(S^{(n)}_k\right)_{k=0,\dots,T_n}$
via
\[S^{(n)}_k:=\left(B^{(n)}_k,V^{(n)}_k\right),\]
where $B^{(n)}_k$ and $V^{(n)}_k$ were defined in equations (\ref{Bv}) and (\ref{cumv}). 
For all $n\in\N$ and $k=1,\dots,T_n$ we set 
\begin{alignat*}{2}
\delta V_k^{(n)}&:=V_k^{(n)}-V_{k-1}^{(n)},\  \qquad\qquad\qquad \delta B_k^{(n)}&&:=B_k^{(n)}-B_{k-1}^{(n)}, \\
\delta \hat{v}^{(n)}_k(x)&:=\E\left(\left.\delta V_k^{(n)}(x)\right|\F^{(n)}_{k-1}\right),\ \  \qquad\delta\hat{B}^{(n)}_k&&:=\E\left(\left.\delta B^{(n)}_k\right|\F^{(n)}_{k-1}\right), \\
\delta \ol{v}^{(n)}_k(x)&:=\delta V_k^{(n)}(x)-\delta \hat{v}^{(n)}_k(x), \quad\qquad\delta\ol{B}^{(n)}_k&&:=\delta B^{(n)}_k-\delta\hat{B}^{(n)}_k.
\end{alignat*}

W.l.o.g.~we will assume that $\left(\Delta x^{(n)}\right)^{-1}\in\N$ for all $n\in\N$.

\s{Fluctuations of the price process}\label{price}

In this section we analyse the fluctuations of the best bid price process $B^{(n)}$. To this end, we introduce a fourth scaling parameter $\Delta p^{(n)}=o(1)$ that controls the proportion of price changes among all events. The scaling limits in \cite{BHQ,HK1,HP} require two time scales, a fast time scale for limit order placements and cancellations and a comparably slow time scale for price changes. The scaling parameter $\Delta p^{(n)}$ introduces the ``slow'' time scale.  

\begin{ass}\label{prob1}
For each $n\in\N$ there exist two functions $p^{(n)}: E_{loc}\ra\R$ and $r^{(n)}: E_{loc}\ra\R_+$ satisfying the boundary condition
\begin{equation}\label{near0}
\left(r^{(n)}(s)\right)^2=\Delta x^{(n)}p^{(n)}(s)\quad\forall\ s=(0,v)\in E_{loc},
\end{equation} 
such that for all $k=1,\dots,T_n$,
\begin{equation}\label{r}
\p\left(\left.\phi_k^{(n)}\in \{A,B\}\ \right| \F_{k-1}^{(n)}\right)=\Delta p^{(n)}\left(r^{(n)}\left(S_{k-1}^{(n)}\right)\right)^2\quad\text{a.s.}
\end{equation}
and
\begin{equation}\label{p}
\p\left(\left.\phi_k^{(n)}=B\ \right| \F_{k-1}^{(n)}\right)-\p\left(\left.\phi_k^{(n)}=A\ \right| \F_{k-1}^{(n)}\right)=\Delta p^{(n)}\Delta x^{(n)}p^{(n)}\left(S_{k-1}^{(n)}\right)\quad\text{a.s.}
\end{equation}
There exists $\eta>0$ such that for all $n\in\N$ and $s\in E_{loc}$,
\begin{equation}\label{pricenonconstant}
r^{(n)}(s)+\left(p^{(n)}(s)\right)^+>\eta.
\end{equation}
\end{ass}

Note that the conditional distribution of the event variables is uniquely determined by equations (\ref{r}) and (\ref{p}).
Moreover, equation (\ref{near0}) guarantees that the price process $B^{(n)}$ will always stay positive. 

The next assumption controls the relative speed at which the different scaling parameters converge to zero. Since the discrete system dynamics are the same as in \cite{HK1}, we must use a different scaling to get a diffusion limit instead of a fluid limit. Intuitively, the average impact of all individual events must be of larger size to generate volatility. By comparing the scaling assumption from \cite{HK1} with Assumption \ref{scaling} below, we see that this is indeed the case.

\begin{ass}\label{scaling}
For all $n\in\N$,
\[\Delta t^{(n)}=\Delta p^{(n)}\left(\Delta x^{(n)}\right)^2=\left(\Delta v^{(n)}\right)^2=o(1).\]
\end{ass}

\begin{rem}
The fact that the conditional distribution of the event variables is uniquely determined by equations (\ref{p}) and (\ref{r}) is different from the corresponding assumption made in \cite{HK1} to derive a large of large numbers in the high frequency regime.  Indeed, while (\ref{r}) can also be found in \cite{HK1}, (\ref{p}) is the only important additional assumption - apart from the different scaling - which is needed to derive a diffusion dynamic for the price process in the high frequency limit. A similar assumption can also be found in \cite{BHQ}.
\end{rem}

The equations (\ref{r}) and (\ref{p}) of Assumption \ref{prob1} yield together with Assumption \ref{scaling} that for all $n\in\N$ and $k\leq T_n$ almost surely
\begin{eqnarray*}
\Delta t^{(n)}\left[r^{(n)}\left(S_{k-1}^{(n)}\right)\right]^2&=&\E\left[\left.\left(\delta B^{(n)}_k\right)^2\right|\F^{(n)}_{k-1}\right],\\
\Delta t^{(n)}p^{(n)}\left(S_{k-1}^{(n)}\right)&=&\E\left[\left.\delta B^{(n)}_k\right|\F^{(n)}_{k-1}\right]=\delta \hat{B}_k^{(n)}.
\end{eqnarray*}
Let us define the process of the (nearly) normalized increments of $B^{(n)}$ as
\begin{equation} \label{Zn}
	\delta Z^{(n)}_k:=\frac{\delta \ol{B}^{(n)}_k}{r^{(n)}\left(S^{(n)}_{k-1}\right)},\quad Z_k^{(n)}:=\sum_{j=1}^k\delta Z^{(n)}_j\quad\text{for all }k=1,\dots,T_n.
\end{equation}

Then we may write for all $n\in\N$,
\begin{equation}\label{SDE_B}
\begin{split}
B^{(n)}(t)&=B_0^{(n)}+\sum_{k=1}^{\lfloor t/\Delta t^{(n)}\rfloor}\delta B_k^{(n)} \\
&=B_0^{(n)}+\sum_{k=1}^{\lfloor t/\Delta t^{(n)}\rfloor} \left[p^{(n)}\left(S^{(n)}_{k-1}\right)\Delta t^{(n)}+r^{(n)}\left(S^{(n)}_{k-1}\right)\delta Z^{(n)}_k\right]
\end{split}
\end{equation}

Through linear interpolation of the $Z^{(n)}_k,\ k=1,\dots,T_n$, we obtain the continuous time process
\begin{equation*} 
	Z^{(n)}(t):=\sum_{k=0}^{T_n}Z_k^{(n)}\1_{\left[t_k^{(n)},t^{(n)}_{k+1}\right)}(t),\quad t\in[0,T].
\end{equation*}

\begin{thm}\label{Z}
Under Assumptions \ref{prob1} and \ref{scaling}, $Z^{(n)}=\left(Z^{(n)}(t)\right)_{t\in[0,T]}$ converges weakly in $\mathcal{D}\left([0,T];\R\right)$ to a standard Brownian motion $Z$ as $n\ra\infty$.
\end{thm}

\begin{proof}
First note that equations (\ref{r}) and (\ref{p}) imply that for all $n\in\N$ and $k\leq T_n$,
\begin{equation}\label{inequality}
-1\leq \frac{\Delta x^{(n)}p^{(n)}\left(S_{k-1}^{(n)}\right)}{\left(r^{(n)}\left(S_{k-1}^{(n)}\right)\right)^2}\leq 1\quad\text{a.s.}
\end{equation}
Moreover, by definition
\begin{equation}\label{pbound}
\Delta t^{(n)}\left|p^{(n)}\left(S_{k-1}^{(n)}\right)\right|\leq \E\left(\left.\left|\delta B_k^{(n)}\right|\right|\F_{k-1}^{(n)}\right)\leq\Delta x^{(n)}\stackrel{n\ra\infty}{\longrightarrow}0\quad\text{a.s.}
\end{equation}
Hence, for any $t\in[0,T]$
\[\sum_{k=1}^{\lfloor t/\Delta t^{(n)}\rfloor}\E\left(\left.\left(\delta Z_k^{(n)}\right)^2\right|\F_{k-1}^{(n)}\right)
=\sum_{k=1}^{\lfloor t/\Delta t^{(n)}\rfloor}\frac{\Delta t^{(n)}\left(r^{(n)}\left(S_{k-1}^{(n)}\right)\right)^2-\left(\Delta t^{(n)}p^{(n)}\left(S_{k-1}^{(n)}\right)\right)^2}{\left(r^{(n)}\left(S_{k-1}^{(n)}\right)\right)^2}\ra t\quad\text{ a.s.}\]
Second, (\ref{pricenonconstant}) and (\ref{inequality}) imply that for all $n\in\N$ and $k\leq T_n$,
\begin{eqnarray*}
\left[r^{(n)}\left(S_{k-1}^{(n)}\right)\right]^{-2}&\leq&\left[r^{(n)}\left(S_{k-1}^{(n)}\right)\right]^{-2} \1_{\left\{r^{(n)}\left(S_{k-1}^{(n)}\right)>\frac{\eta}{2}\right\}}+\left[r^{(n)}\left(S_{k-1}^{(n)}\right)\right]^{-2}\1_{\left\{\left(p^{(n)}\left(S_{k-1}^{(n)}\right)\right)^+>\frac{\eta}{2}\right\}} \\
&\leq&\frac{4}{\eta^2}+\left[r^{(n)}\left(S_{k-1}^{(n)}\right)\right]^{-2}\1_{\left\{\left[r^{(n)}\left(S_{k-1}^{(n)}\right)\right]^2>\Delta x^{(n)}\frac{\eta}{2}\right\}}\\
&\leq&\frac{4}{\eta^2}+\frac{2}{\Delta x^{(n)}\eta}\quad\text{a.s.}
\end{eqnarray*}
Therefore, there exists a deterministic sequence $(c_n)$ converging to zero such that for all $k=1,\dots,T_n$,
\begin{equation}\label{Zbound}
\begin{split}
\left|\delta Z^{(n)}_k\right|^2&=\frac{\left[\Delta x^{(n)}\left(\1_B\left(\phi_k^{(n)}\right)-\1_A\left(\phi_k^{(n)}\right)\right)-\Delta t^{(n)} p^{(n)}\left(S_{k-1}^{(n)}\right)\right]^2}{\left[r^{(n)}\left(S_{k-1}^{(n)}\right)\right]^2} \\
&\leq2\left[\left(\Delta x^{(n)}\right)^2+\left(\Delta t^{(n)} p^{(n)}\left(S_{k-1}^{(n)}\right)\right)^2\right]\frac{2}{\eta}\left(\frac{2}{\eta}+\frac{1}{\Delta x^{(n)}}\right) \\
&\leq c_n\quad\text{ a.s.}
\end{split}
\end{equation}
We conclude that for all $\eps>0$,
\begin{eqnarray*}
\sum_{k=1}^{\lfloor t/\Delta t^{(n)}\rfloor}\E\left(\left|\delta Z^{(n)}_k\right|^2\1_{\left\{\left|\delta Z^{(n)}_k\right|>\eps\right\}}\right)&\leq& \frac{c_n}{\eps^2}\sum_{k=1}^{\lfloor t/\Delta t^{(n)}\rfloor}\E\left|\delta Z^{(n)}_k\right|^2\leq \frac{t}{\eps^2}\cdot c_n\ra0,
\end{eqnarray*}
i.e. the Lindeberg condition is satisfied. Therefore, the functional central limit theorem for martingale difference arrays (cf.~Theorem 18.2 in \cite{Billingsley}) implies that $Z^{(n)}$ converges weakly to a standard Brownian motion. 
\end{proof}

In order to obtain the convergence of the full price process in Section 5 below we also have to assume that the drift and volatility functions $p^{(n)}$ and $r^{(n)},\ n\in\N,$ satisfy a continuity condition and that they converge to some functions $p$ and $r$ as $n\ra\infty$. 

\begin{ass}\phantomsection\label{prob2}
\be[label={(\roman*)},ref={\thecor(\roman*)}]
\item\label{prob21}
There exist functions $p: E_{loc}\ra\R,\ r: E_{loc}\ra\R_+,$ and $C<\infty$ such that for all $s=(b,v),\tilde{s}=(\tilde{b},\tilde{v})\in E_{loc}$,
\begin{equation*}\label{lineargrowth}
\left|p\left(s\right)\right|+r\left(s\right)\leq C(1+|b|)
\end{equation*}
and for all $m\in\N$,
\begin{equation*}\label{pq}
\sup_{s=(b,v)\in E_{loc}}\left|p^{(n)}\left((b\wedge m,v)\right)-p\left((b\wedge m,v)\right)\right|+\left|r^{(n)}\left((b\wedge m,v)\right)-r\left((b\wedge m,v)\right)\right|\ra0.
\end{equation*}
\item\label{prob22}
There exists $L<\infty$ such that for all $n\in\N$ and $s=(b,v),\tilde{s}=(\tilde{b},\tilde{v})\in E_{loc}$,
\begin{equation*}\label{crazy}
\begin{split}
&\max\left\{\left|p^{(n)}(s)-p^{(n)}(\tilde{s})\right|,\left|r^{(n)}(s)-r^{(n)}(\tilde{s})\right|\right\}\\ 
&\qquad\qquad\leq L\left(1+|b|+|\tilde{b}|\right)\left(1+\left\Vert v\1_{[0,b\vee\tilde{b}]}\right\Vert_{L^2}+\left\Vert \tilde{v}\1_{[0,b\vee\tilde{b}]}\right\Vert_{L^2}\right)\left\{\left|b-\tilde{b}\right|+\left\Vert\left(v-\tilde{v}\right)\1_{[0,b\vee\tilde{b}]}\right\Vert_{L^2}\right\}.
\end{split}
\end{equation*}
\ee
\end{ass}

Assumption \ref{prob22} is similar to a local Lipschitz assumption. It will play a key role in the proof of the main theorem later on. The following example illustrates the assumed dependence structure.

\begin{ex}\label{ex0}
In order to model dependence on standing volumes we can integrate a Lipschitz continuous function $h:\R\ra\R$ against cumulated volumes standing to the left of the price process. If we suppose that $h$ has compact support in $\R_-$, then for all $s=(b,v),\tilde{s}=(\tilde{b},\tilde{v})\in E_{loc}$,
\begin{eqnarray*}
&&\left|\left\langle v(\cdot+b)\1_{[-b,0]},h\right\rangle-\left\langle \tilde{v}\left(\cdot+\tilde{b}\right)\1_{[-\tilde{b},0]},h\right\rangle\right|
=\left|\left\langle v,h(\cdot-b)\1_{[0,b]}\right\rangle-\left\langle \tilde{v},h\left(\cdot-\tilde{b}\right)\1_{[0,\tilde{b}]}\right\rangle\right|\\
&\leq& \left|\left\langle v-\tilde{v},h\left(\cdot-\tilde{b}\right)\1_{[0,\tilde{b}]}\right\rangle\right|+\left|\left\langle v\1_{[0,b\vee\tilde{b}]},h(\cdot-b)-h\left(\cdot-\tilde{b}\right)\right\rangle\right|\\
&\leq& \left\Vert h\right\Vert_{L^2}\cdot\left\Vert (v-\tilde{v})\1_{[0,b\vee \tilde{b}]}\right\Vert_{L^2}+\left\Vert v\1_{[0,b\vee\tilde{b}]}\right\Vert_{L^2} \cdot  L\left\Vert \1_{[0,b\vee\tilde{b}]}\left(b-\tilde{b}\right)\right\Vert_{L^2}\\
&\leq& \left\Vert h\right\Vert_{L^2}\cdot\left\Vert (v-\tilde{v})\1_{[0,b\vee \tilde{b}]}\right\Vert_{L^2}+ L\left\Vert v\1_{[0,b\vee\tilde{b}]}\right\Vert_{L^2}\left(1+|b|+|\tilde{b}|\right)\left|b-\tilde{b}\right|.
\end{eqnarray*}
Now if $P,R$ are Lipschitz continous functions, we may define for all $s=(b,v)\in E_{loc}$,
\[p^{(n)}\left(s\right):=P\left(\left\langle v(\cdot+b)\1_{[-b,0]},h\right\rangle\right),\qquad r^{(n)}\left(s\right):=R\left(\left\langle v(\cdot+b)\1_{[-b,0]},h\right\rangle\right)\]
and the so defined functions $p^{(n)}$ and $r^{(n)}$ satisfy Assumption \ref{prob22}. 
\end{ex}

\s{Fluctuations of the volume process}\label{volume}

In this section we analyze the fluctuation of the infinite dimensional volume process $V^{(n)}$. In a first step we compute its conditional moments and prove their convergence as $n \ra \infty$. Subsequently, we represent it as the solution to a stochastic differential equations driven by infinite dimensional martingale that converges in distribution to a cylindrical Brownian motion as $n \to \infty$. Since $V^{(n)}$ is not an $L^2$-valued process, but only $L^2_{loc}$-valued, we need to localize the analysis. 

We make the following assumption on the joint distribution of the random variables $\omega_k^{(n)}$ and $\pi_k^{(n)}$.

\begin{ass}\label{density}
There exists an $M>0$ such that for all $n\in\N$ and $k\leq T_n$,
\begin{equation}\label{wm}
\p\left(\omega_k^{(n)}\in[-M,M],\ \pi_k^{(n)}\in[0,\infty)\right)=1.
\end{equation}
For every $n\in\N$ there exist two measurable functions $g^{(n)},\ h^{(n)}: E_{loc}\times\R_+\ra\R_+$
such that for all $k=1,\dots,T_n$ and all $D\in\mathcal{B}(\R_+)$,
\[\E\left(\left.\left(\omega_k^{(n)}\right)^2\1_C\left(\phi_k^{(n)}\right)\1_D\left(\pi_k^{(n)}\right)\ \right|\ \F_{k-1}^{(n)}\right)=\int_Dg^{(n)}\left(S_{k-1}^{(n)};y\right)dy\quad\text{a.s.}\] 
and
\[\E\left(\left.\omega_k^{(n)}\1_C\left(\phi_k^{(n)}\right)\1_D\left(\pi_k^{(n)}\right)\ \right|\ \F_{k-1}^{(n)}\right)=\Delta v^{(n)}\int_Dh^{(n)}\left(S_{k-1}^{(n)};y\right)dy\quad\text{a.s.}\] 
\end{ass}

According to Assumptions \ref{prob1} and \ref{density} the process $\left(S_k^{(n)}\right)_{k=0,\dots,T_n}$ is a homogeneous Markov chain for each $n\in\N$. Furthermore, (\ref{wm}) and Assumption \ref{scaling} imply that for all $m>0$, $n\in\N$, and $k\leq T_n$,
\begin{eqnarray*}
\left\Vert \delta V^{(n)}_k\1_{[0,m]}\right\Vert^2_{L^2}&\leq& \left(\Delta v^{(n)}\right)^2M^2
\left\Vert\sum_{j=0}^{\lfloor\cdot/\Delta x^{(n)}\rfloor}\1_{I^{(n)}\left(\pi^{(n)}_k\right)}\left(x_j^{(n)}\right)\1_{[0,m]}\right\Vert_{L^2}^2 \\
&\leq& \Delta t^{(n)}M^2m\quad\text{a.s.}
\end{eqnarray*}
and therefore for all $m>0$ also
\begin{equation}\label{vbound}
\begin{split}
\left\Vert\delta\ol{v}^{(n)}_k\1_{[0,m]}\right\Vert_{L^2}^2&\leq \left\Vert\delta\hat{v}^{(n)}_k\1_{[0,m]}\right\Vert_{L^2}^2+
\left\Vert \delta V_k^{(n)}\1_{[0,m]}\right\Vert^2_{L^2}\\
&\leq \E\left(\left.\left\Vert \delta V_k^{(n)}\1_{[0,m]}\right\Vert^2_{L^2}\right|\F^{(n)}_{k-1}\right)+\left\Vert \delta V_k^{(n)}\1_{[0,m]}\right\Vert^2_{L^2} \\
&\leq 2M^2m\Delta t^{(n)}\quad\text{a.s.}
\end{split}
\end{equation}

The next two assumptions deal with the convergence and continuity of $g^{(n)}$ and $h^{(n)}$.

\begin{ass}\phantomsection\label{g}
\be[label={(\roman*)},ref={\thecor(\roman*)}]
\item\label{gi}
There exists a measurable function $g: E_{loc}\times\R_+\ra \R_+$ satisfying 
\begin{equation*}\label{qis}
\inf_{s\in E_{loc}}g(s;y)>0\quad\forall\ y\in\R_+ 
\end{equation*}
such that
\begin{equation*}\label{gconv}
\sup_{s\in E_{loc}}\int_0^\infty\left|g^{(n)}(s;y)-g(s;y)\right|dy\ra0.
\end{equation*}
\item\label{gii}There exists an $L<\infty$ such that for all $n\in\N$ and $s=(b,v),\tilde{s}=(\tilde{b},\tilde{v})\in E_{loc}$,
\begin{equation*}\label{g1}
\begin{split}
&\int_0^\infty\left|g^{(n)}(s;y)-g^{(n)}(\tilde{s};y)\right|dy\\
&\qquad\qquad\leq L\left(1+|b|+]\tilde{b}]\right)\left(1+\left\Vert v\1_{[0,b\vee\tilde{b}]}\right\Vert_{L^2}+\left\Vert \tilde{v}\1_{[0,b\vee\tilde{b}]}\right\Vert_{L^2}\right)\left\{\left|b-\tilde{b}\right|+\left\Vert\left(v-\tilde{v}\right)\1_{[0,b\vee\tilde{b}]}\right\Vert_{L^2}\right\}.
\end{split}
\end{equation*}
\ee
\end{ass}

The next assumption is key to the derivation of a diffusion limit for the $L^2_{loc}$-valued functions $V^{(n)}$. It states that order placements and cancellations are expected to be approximately of the same size and that the expected disbalance between both also scales in $n$. This guarantees that the cumulated volume process will not explode when passing to the scaling limit. 

\begin{ass}\phantomsection\label{h}
\be[label={(\roman*)},ref={\thecor(\roman*)}]
\item\label{hi}
There exists a measurable function $h:E_{loc}\times \R_+\ra\R$ satisfying 
\begin{equation*}\label{hintbar}
\sup_{s\in E_{loc}}\int_0^\infty \left|h(s;y)\right|^2dy<\infty
\end{equation*}
such that
\begin{equation*}\label{hconv}
\sup_{s\in E_{loc}}\int_0^\infty \left|h^{(n)}(s;y)-h(s;y)\right|^2dy\ra0.
\end{equation*} 
\item\label{hii}There exists an $L<\infty$ such that for all $n\in\N$ and $s=(b,v),\tilde{s}=\left(\tilde{b},\tilde{v}\right)\in E_{loc}$,
\begin{equation*}\label{h1}
\begin{split}
&\left(\int_0^\infty\left|h^{(n)}(s;y)-h^{(n)}(\tilde{s};y)\right|^2dy\right)^{1/2}\\
&\qquad\qquad\leq L\left(1+|b|+|\tilde{b}|\right)\left(1+\left\Vert v\1_{[0,b\vee\tilde{b}]}\right\Vert_{L^2}+\left\Vert \tilde{v}\1_{[0,b\vee\tilde{b}]}\right\Vert_{L^2}\right)\left\{\left|b-\tilde{b}\right|+\left\Vert\left(v-\tilde{v}\right)\1_{[0,b\vee\tilde{b}]}\right\Vert_{L^2}\right\}.
\end{split}
\end{equation*}
\ee
\end{ass}

\subsection{Basis functions}

Our goal is to represent the volume function as a stochastic differential equation driven by an infinite dimensional martingale whose increments are orthogonal across different basis functions of $L^2(\R_+;\R)$. We choose the Haar basis, i.e.~we specify the basis functions $(f_i)$ as follows: for each $k\in\N_0$ we set $g^k_{-1}(x)=\1_{[k,k+1)}(x)$. Moreover, we set for all $k,l\in\N_0$ ,
\[g^k_l(x):=\begin{cases}2^{l/2}&:\ x\in\left[k2^{-l},\left(k+\frac{1}{2}\right)2^{-l}\right)\\-2^{l/2}&:\ x\in\left[\left(k+\frac{1}{2}\right)2^{-l},(k+1)2^{-l}\right)\\0&:\ \text{else}\end{cases}.\]
To define the $(f_i)$ we now reorder the $(g^k_l)$ in a diagonal procedure:
\[f_1:=g^0_{-1},\ f_2:=g^1_{-1},\ f_3:=g^0_0,\ f_4:=g^2_{-1},\ f_5:=g^1_0,\ f_6:=g^0_1,\ \dots\]

In the following we denote by $k(i)\in\N_0$ and $l(i)\in\N_{-1}:=\N_0\cup\{-1\}$ the indeces such that $f_i\equiv g^{k(i)}_{l(i)}$.

Let us define for each $i\in\N$ the functions $F_i:\R_+\ra\R$ and $F_i^{(n)}:\R_+\ra\R,\ n\in\N,$ via
\[F_i(y):=\int_y^\infty f_i(x)dx,\qquad F_i^{(n)}(y):=\int_{\Delta x^{(n)}\left\lfloor y/\Delta x^{(n)}\right\rfloor}^\infty f_i(x)dx.\]
We shall see that the drift and the volatility of the volume processes can be expressed in terms of the functions $F_i$ and $F^{(n)}_i$. 
We notice that $|F_i(y)|\vee\left|F_i^{(n)}(y)\right|\leq 1$ for all $y\in\R_+$ and $i,n\in\N$. In addition, we will often use the fact that if $l(i)\geq0$, then $$supp(F_i)=\left[k(i)2^{-l(i)},(k(i)+1)2^{-l(i)}\right],$$ i.e.~$\left|supp(F_i)\right|\leq1$. Similarly, also $\left|supp\left(F^{(n)}_i\right)\right|\leq1$ for all $i,n\in\N$ with $l(i)\geq0$. We also notice that if $l(i)=-1$, then $supp(F_i)=supp\left(F_i^{(n)}\right)=[0,k(i)+1]$.
Moreover, we have the  $L^2_{loc}$-representation
\[\1_{[y,\infty)}(x)=\sum_iF_i(y)f_i(x),\qquad \1_{\left[\Delta x^{(n)}\lfloor y/\Delta x^{(n)}\rfloor,\infty\right)}(x)=\sum_iF_i^{(n)}(y)f_i(x).\]
Finally, for all $m\in\N$ we define the index set
\begin{equation}\label{im}
\mathcal{I}_m:=\{i\in\N:\ supp(f_i)\cap(0,m)\neq\emptyset\}.
\end{equation}
Note that for all $m\in\N$, $(f_i)_{i\in\mathcal{I}_m}$ is a basis of $L^2([0,m])$.
Furthermore, for all $n,m\in\N$ and $y\in\R_+$,
\[\sum_{i\in\mathcal{I}_m}\left[F^{(n)}_i(y)\right]^2\leq m\qquad \text{and}\qquad\sum_{i\in\mathcal{I}_m}\left[F_i(y)\right]^2\leq m.\]

We shall repeatedly use the following technical lemma. It allows us to approximate the conditional moments of volume increments using finitely many basis functions after localisation. 

\begin{lem}\label{eps}
For each $\eps>0$ and $m\in\N$ there exists a finite subset $J\subset\mathcal{I}_m$ such that for all $y\in\R_+$,
\[\sum_{i\in\mathcal{I}_m\backslash J}\left(F_i(y)\right)^2\leq\eps\qquad\text{and}\qquad\sum_{i\in\mathcal{I}_m\backslash J}\left(F^{(n)}_i(y)\right)^2\leq\eps\quad \forall\ n\in\N.\]
\end{lem}

\begin{proof}
For fixed $\eps>0$ and $m\in\N$ set $l_0:=\min\left\{l\in\N:\ 2^{-l}\leq\eps\right\}$ and $J:=\left\{i\in\mathcal{I}_m:\ l(i)\leq l_0\right\}$. Now note that for all $i\in\N$,
\[|F_i(y)|=\left|\int_y^\infty f_i(x)dx\right|\leq 2^{-l(i)/2}\quad\forall\ y\in\R_+.\]
Furthermore for every $l\in\N$ and $y\in\R_+$ there exists exactly one $i\in\N$ with $l(i)=l$ such that $F_i(y)\neq0$. Therefore,
\[\sum_{i\in\mathcal{I}_m\backslash J}\left(F_i(y)\right)^2\leq \sum_{l>l_0}2^{-l(i)}=2^{-l_0}\leq\eps\quad\forall\ y\in\R_+.\]
Since this is true for all $y\in\R_+$, it is also true for all $\Delta x^{(n)}\lfloor y/\Delta x^{(n)}\rfloor$ with $n\in\N$ and $y\in\R_+$. Hence,
\[\sum_{i\in\mathcal{I}_m\backslash J}\left(F_i^{(n)}(y)\right)^2\leq\eps\quad\forall\ y\in\R_+,\ n\in\N.\]
\end{proof}

\subsection{Convergence of drift, volatility and correlation functions}\label{musigma}
We are now going to analyse the convergence of the conditional expectations and variances of the volume increments. It will turn out that in the limit they can be described in terms of the functions $\mu_i:E_{loc}\ra\R$ and $\sigma_i:E_{loc}\ra\R_+$  $(i \in \N)$ defined by:
\[\mu_i(s):=\int_0^\infty h(s;y)F_i(y)dy,\qquad\left(\sigma_i(s)\right)^2:=\int_0^\infty g(s;y)\left[F_i(y)\right]^2dy.\]

\begin{lem}\label{sigpos} 
Given Assumption \ref{gi} we have for all $i\in\N$, $\inf_{s\in E_{loc}}\sigma_i(s)>0$. 
\end{lem}

\begin{proof}
By definition $F_i(y)\neq0$ for all $y\in\left(k(i)2^{-l(i)},(k(i)+1)2^{-l(i)}\right)$. Thus, the claim follows from the fact that $g(\cdot;y)$ is bounded away from zero for each $y\in\R_+$ according to Assumption \ref{gi}.
\end{proof}

In view of the preceding lemma we can define for all $i,j\in\N$ the function $\rho_{ij}:E_{loc}\ra[-1,1]$ via
\[\sigma_i(s)\sigma_j(s)\rho_{ij}(s):=\int_0^\infty g(s;y)F_i(y)F_j(y)dy.\]

Moreover, we define for each $n,i,j\in\N$ the following functions from $E_{loc}$ to $\R$,
\begin{eqnarray*}
\mu^{(n)}_i(s)&:=&\int_0^\infty h^{(n)}(s;y)F^{(n)}_i(y)dy,\\
\sigma^{(n)}_i(s)&:=&\left(\int_0^\infty g^{(n)}(s;y)\left[F^{(n)}_i(y)\right]^2dy
-\Delta t^{(n)}\left(\mu_i^{(n)}(s)\right)^2\right)^{1/2},\\
\rho^{(n)}_{ij}(s)&:=&\frac{\1_{(0,\infty)}\left(\sigma_i^{(n)}(s)\sigma_j^{(n)}(s)\right)}{\sigma_i^{(n)}(s)\sigma^{(n)}_j(s)}
\left(\int_0^\infty g^{(n)}(s;y)F^{(n)}_i(y)F^{(n)}_j(y)dy
-\Delta t^{(n)}\mu_i^{(n)}(s)\mu^{(n)}_j(s)\right),
\end{eqnarray*}
and the $L^2_{loc}(\R_+)$-valued functions
\[\mu^{(n)}(s;\cdot):=\sum_i\mu_i^{(n)}(s)f_i(\cdot)\qquad\text{and}\qquad\mu(s;\cdot):=\sum_i\mu_i(s)f_i(\cdot).\]

Note that with this notation we have for all $x\in\R_+$ and $n\in\N$, making use of Assumption \ref{scaling},
\begin{equation}\label{1moment}
\begin{split}
\delta\hat{v}_k^{(n)}(x)&=\Delta v^{(n)}\E\left(\left.\Delta x^{(n)}\sum_{j=0}^{\lfloor x/ \Delta x^{(n)}\rfloor}M_k^{(n)}\left(x_j^{(n)}\right)\right|\F^{(n)}_{k-1}\right)\\
&=\Delta v^{(n)}\int_0^\infty \Delta v^{(n)}h^{(n)}\left(S_{k-1}^{(n)};y\right) \sum_{j=0}^{\lfloor x/\Delta x^{(n)}\rfloor}\1_{I^{(n)}(y)}\left(x_j^{(n)}\right)dy\\
&=\Delta t^{(n)}\int_0^\infty h^{(n)}\left(S_{k-1}^{(n)};y\right)  \1_{\left[\Delta x^{(n)}\lfloor y/\Delta x^{(n)}\rfloor,\infty\right)}(x) dy=\Delta t^{(n)}\mu^{(n)}\left(S_{k-1}^{(n)};x\right)
\end{split}
\end{equation}
as well as
\begin{equation*}
\begin{split}
\E\left(\left.\left\langle\delta V_k^{(n)},f_i\right\rangle^2\right| \F_{k-1}^{(n)}\right) &=\Delta t^{(n)}\E\left(\left.\1_C\left(\phi_k^{(n)}\right)\left[\omega_k^{(n)}\int_{\R_+} f_i(x)\sum_{j=0}^{\lfloor x/\Delta x^{(n)}\rfloor}\1_{I^{(n)}\left(\pi_k^{(n)}\right)}\left(x_j^{(n)}\right)dx\right]^2\right| \F_{k-1}^{(n)}\right)\\
&=\Delta t^{(n)}\E\left(\left.\1_C\left(\phi_k^{(n)}\right)\left(\omega_k^{(n)}\right)^2\left[F_i^{(n)}\left(\pi_k^{(n)}\right)\right]^2\right| \F_{k-1}^{(n)}\right) \\
&=\Delta t^{(n)}\int_0^\infty g^{(n)}\left(S_{k-1}^{(n)};y\right)\left[F_i^{(n)}(y)\right]^2dy\qquad\qquad\qquad\qquad\qquad\qquad\qquad\qquad\ \\
&=\Delta t^{(n)}\left[\left(\sigma_i^{(n)}\left(S_{k-1}^{(n)}\right)\right)^2+\Delta t^{(n)}\left(\mu_i^{(n)}\left(S_{k-1}^{(n)}\right)\right)^2\right],\qquad\qquad\qquad\qquad\qquad\qquad\qquad\qquad\qquad\qquad\ \ \ 
\end{split}
\end{equation*}
i.e.
\begin{equation}\label{2moment}
\E\left(\left.\left\langle\delta \ol{v}_k^{(n)},f_i\right\rangle^2\right| \F_{k-1}^{(n)}\right)=\Delta t^{(n)}\left(\sigma_i^{(n)}\left(S_{k-1}^{(n)}\right)\right)^2.
\end{equation}

Similar calculations show that
\begin{equation}\label{cc}
\rho^{(n)}_{ij}\left(S_{k-1}^{(n)}\right)=\frac{\E\left(\left.\left\langle \delta\ol{v}^{(n)}_k,f_i\right\rangle\left\langle \delta\ol{v}^{(n)}_k,f_j\right\rangle\right|\F_{k-1}^{(n)}\right)}{\sigma_i^{(n)}\left(S_{k-1}^{(n)}\right)\sigma^{(n)}_j\left(S_{k-1}^{(n)}\right)}\1_{(0,\infty)}\left(\sigma_i^{(n)}\left(S_{k-1}^{(n)}\right)\sigma_j^{(n)}\left(S_{k-1}^{(n)}\right)\right).
\end{equation}

The next three lemmata establish the convergence of the drift, the volatility and the covariance functions introduced above. 

\begin{lem}\label{mu} 
	Given Assumption \ref{hi} we have for all $m\in\N$,
\[\sup_{s\in E_{loc}}\left\Vert \mu(s)\1_{[0,m]}\right\Vert_{L^2}<\infty \qquad\text{and}\qquad\sup_{s\in E_{loc}}\left\Vert\left(\mu^{(n)}(s)-\mu(s)\right)\1_{[0,m]}\right\Vert_{L^2}\ra0.\]
\end{lem}

\begin{proof}
Since $supp(F_i)\subset[0,m]$ for all $i\in\mathcal{I}_m$,
\begin{eqnarray*}
\sup_{s\in E_{loc}}\left\Vert \mu(s)\1_{[0,m]}\right\Vert_{L^2}^2&=&
\sup_{s\in E_{loc}}\sum_{i\in\mathcal{I}_m}\left(\int_0^\infty h(s;y)F_i(y)dy\right)^2\\
&\leq&\sup_{s\in E_{loc}}m\sum_{i\in\mathcal{I}_m}\int_0^m \left(h(s;y)\right)^2\left(F_i(y)\right)^2dy \\
&\leq& m^2\cdot\sup_{s\in E_{loc}}\int_0^\infty \left(h(s;y)\right)^2dy<\infty.
\end{eqnarray*}
By a similar reasoning we can estimate for all $m\in\N$,
\begin{eqnarray*}
& & \sup_{s\in E_{loc}}\sum_{i\in\mathcal{I}_m}\left(\int_0^\infty h(s;y)\left(F_i^{(n)}(y)-F_i(y)\right)dy\right)^2 \\
&\leq& \sup_{s\in E_{loc}}m\sum_{i\in\mathcal{I}_m}\int_0^\infty \left(h(s;y)\right)^2\left(F_i^{(n)}(y)-F_i(y)\right)^2dy\\
&=&\ m\cdot \sup_{s\in E_{loc}}\int_0^\infty \left(h(s;y)\right)^2\left\Vert\1_{\left[\Delta x^{(n)}\lfloor y/\Delta x^{(n)}\rfloor ,y\right]}(\cdot)\right\Vert^2_{L^2}dy \\ &\leq& m\Delta x^{(n)}\sup_{s\in E_{loc}}\int_0^\infty \left(h(s;y)\right)^2dy \ra 0
\end{eqnarray*}
and by Assumption \ref{hi} also
\begin{eqnarray*}
&&\sup_{s\in E_{loc}}\sum_{i\in\mathcal{I}_m}\left(\int_0^\infty\left(h^{(n)}(s;y)-h(s;y)\right)F^{(n)}_i(y)dy\right)^2\\
&\leq& \sup_{s\in E_{loc}}m\sum_{i\in\mathcal{I}_m}\int_0^\infty\left(h^{(n)}(s;y)-h(s;y)\right)^2\left(F^{(n)}_i(y)\right)^2dy\\
&\leq&m^2\sup_{s\in E_{loc}}\int_0^\infty\left(h^{(n)}(s;y)-h(s;y)\right)^2dy\ra0.
\end{eqnarray*}
\end{proof}

\begin{lem}\label{sigma2}
Given Assumptions \ref{density}, \ref{gi}, and \ref{hi} we have for all $m\in\N$,  
\[\sup_{s\in E_{loc}}\sum_{i\in\mathcal{I}_m}\left(\sigma^{(n)}_i(s)\right)^2\leq mM^2\qquad\text{and}\qquad\sup_{s\in E_{loc}}\sum_{i\in\mathcal{I}_m}\left|\sigma_i^{(n)}(s)-\sigma_i(s)\right|^2\ra0.\]
\end{lem}

\begin{proof}
First, it follows from Assumption \ref{density} and equation (\ref{2moment}) that for all $m\in\N$ and $s\in E_{loc}$,
\[\sum_{i\in\mathcal{I}_m}\left(\sigma_i(s)\right)^2\leq\sum_{i\in\mathcal{I}_m}\int_0^\infty g(s;y)\left[F_i(y)\right]^2dy\leq mM^2.\]
Second, by Assumption \ref{gi} for all $m\in\N$,
\[\sup_{s\in E_{loc}}\sum_{i\in\mathcal{I}_m}\left|\int_0^\infty\left(g^{(n)}(s;y)-g(s;y)\right)\left[F_i(y)\right]^2dy\right|\leq m\cdot\sup_{s\in E_{loc}}\int_0^\infty\left|g^{(n)}(s;y)-g(s;y)\right|dy\ra 0\]
and it follows from Lemma \ref{mu} that for all $m\in\N$,
\[\Delta t^{(n)}\sup_{s\in E}\sum_{i\in\mathcal{I}_m}\left(\mu_i^{(n)}(s)\right)^2\ra0.\]
Next fix $m\in\N$ and let $\eps>0$. By Lemma \ref{eps} we find a finite subset $J\subset\mathcal{I}_m$ such that for all $n\in\N$ and $y\in\R_+$,
\[\sum_{i\in\mathcal{I}_m\backslash J}\left[F_i^{(n)}(y)\right]^2\leq\frac{\eps}{4M^2}\qquad\text{and}\qquad\sum_{i\in\mathcal{I}_m\backslash J}\left[F_i(y)\right]^2\leq\frac{\eps}{4M^2}.\]
Now we choose $n_0=n_0(\eps,m)$ such that for all $i\in\N$, $y\in\R_+$, and $n\geq n_0$,
\[\left|\left[F_i^{(n)}(y)\right]^2-[F_i(y)]^2\right|\leq 2\left|F_i^{(n)}(y)-F_i(y)\right|
\leq 2\left\Vert \1_{\left[\Delta x^{(n)}\lfloor y/\Delta x^{(n)}\rfloor,y\right]}\right\Vert_{L^2} \leq 2\left(\Delta x^{(n)}\right)^{1/2}<\frac{\eps}{2M^2|J|}.\]
We deduce that for all $n\geq n_0$ and $s\in E_{loc}$,
\begin{eqnarray*}
&&\sum_{i\in\mathcal{I}_m}\left|\int_0^\infty g^{(n)}(s;y)\left(\left[F_i^{(n)}(y)\right]^2-[F_i(y)]^2\right)dy\right|\\
&\leq&\frac{\eps}{2M^2}\int_0^\infty g^{(n)}(s;y)dy+\int_0^\infty g^{(n)}(s;y)\sum_{i\in J}\left|\left[F_i^{(n)}(y)\right]^2-[F_i(y)]^2\right|dy\\
&<& \frac{\eps}{2}+\frac{\eps}{2M^2}\int_0^\infty g^{(n)}(s;y)dy\leq\eps.
\end{eqnarray*}
Therefore, we have,  
\begin{eqnarray*}
& & \lim_{n\ra\infty}\sup_{s\in E_{loc}}\sum_{i\in\mathcal{I}_m}\left|\sigma_i^{(n)}(s)-\sigma_i(s)\right|^2\\
&\leq&\lim_{n\ra\infty}\sup_{s\in E_{loc}}\sum_{i\in\mathcal{I}_m}\left|\left(\sigma_i^{(n)}(s)\right)^2-\left(\sigma_i(s)\right)^2\right|\qquad\qquad\qquad\qquad\qquad\qquad\\
&\leq&\lim_{n\ra\infty}\sup_{s\in E_{loc}}\sum_{i\in\mathcal{I}_m}\left|\left(\sigma_i^{(n)}(s)\right)^2+\Delta t^{(n)}\left(\mu_i^{(n)}(s)\right)^2-\left(\sigma_i(s)\right)^2\right|+\lim_{n\ra\infty}\Delta t^{(n)}\sup_{s\in E_{loc}}\sum_{i\in\mathcal{I}_m}\left(\mu_i^{(n)}(s)\right)^2 
=0.
\end{eqnarray*}
\end{proof}

\begin{lem}\label{rho1}
Given Assumptions \ref{scaling}, \ref{density}, \ref{gi}, and \ref{hi} we have for all $i,j\in\N$,
\[\sup_{s\in E_{loc}}\left|\rho_{ij}^{(n)}(s)-\rho_{ij}(s)\right|\ra0.\]
\end{lem}

\begin{proof}
One can show similary to the proof of Lemma \ref{sigma2} that for every fixed $i,j\in\N$,
\[\rho^{(n)}_{ij}(s)\sigma^{(n)}_i(s)^{(n)}\sigma_j(s)=\int_0^\infty g^{(n)}(s;y)F^{(n)}_i(y)F^{(n)}_j(y)dy-\Delta t^{(n)}\mu_i^{(n)}(s)\mu_j^{(n)}(s)\]
converges to $\rho_{ij}(s)\sigma_i(s)\sigma_j(s)$ uniformly in $s\in E_{loc}$. Since $\sigma_i^{(n)}$ and $\sigma_j^{(n)}$ converge to $\sigma_i$ respectively $\sigma_j$ uniformly by Lemma \ref{sigma2} and since both, $\sigma_i$ and $\sigma_j$ are uniformly bounded from below by Lemma \ref{sigpos}, the claim follows.
\end{proof}

\subsection{Orthogonal decomposition}\label{c}

In order to identify the volume as the solution of some stochastic differential equation we need to decorrelate the normalised volume increments. To this end, we introduce in this subsection an orthogonal decomposition of the increments using the algorithm from Appendix \ref{decomposition}. We assume that the probability spaces are rich enough to support  i.i.d. Bernoulli random variables. 

\begin{ass}\label{U}
For every $n\in\N$ there exists a field of i.i.d.~random variables $\left(U^{(n),i}_k\right)_{k,i\in\N}$ on $\left(\Omega^{(n)},\F^{(n)},\p^{(n)}\right)$, which are independent of $S^{(n)}$, such that
\[\p\left(U^{(n),i}_k=-1\right)=\p\left(U^{(n),i}_k=1\right)=\frac{1}{2}.\]
\end{ass}

We recall the definition of the (conditional) correlation coefficients $\rho^{(n)}_{ij}(\cdot)$ $(n,i,j \in \N, j \leq i)$ from (\ref{cc}). The algorithm in Appendix \ref{decomposition} provides, for each $n \in \N$, an array $\left(c^{(n)}_{ij}(\cdot) \right)_{j \leq i}$ of measurable functions from $E_{loc}$ to $[-1,1]$ together with the ``inverse array'' $\left( \alpha^{(n)}_{ij}(\cdot) \right)_{j \leq i}$ in terms of the Borel measurable correlation coefficients $\left( \rho_{ij}^{(n)}(\cdot) \right)_{j, i}$. Now if we define for any $n,i \in \N$ and $k\leq T_n$ the random variables 
\[
	Z^{(n),i}_k:=\begin{cases}\frac{\left\langle\delta \ol{v}_k^{(n)},f_i\right\rangle}{\sigma_i^{(n)}\left(S^{(n)}_{k-1}\right)}&:\ \sigma_i^{(n)}\left(S_{k-1}^{(n)}\right)>0\\ \left(\Delta t^{(n)}\right)^{1/2}U_k^{(n),i}&:\ \sigma_i^{(n)}\left(S_{k-1}^{(n)}\right)=0\end{cases},
\]
then the conditional correlation between $Z^{(n),i}_k$ and $Z^{(n),j}_k$ is precisely $\rho_{ij}\left(S_{k-1}^{(n)}\right)$.
If we now define, for each $n\in\N$ and $k\leq T_n$, a sequence of random variables $\delta W^{(n),i}_k,\ i\in\N,$  inductively via
\[
	\delta W^{(n),1}_k:=Z^{(n),1}_k\left(S^{(n)}_{k-1}\right)
\]	
and for all $i>1$,
\begin{equation}\label{deltaW}
	\delta W^{(n),i}_k:=\begin{cases}\frac{1}{c^{(n)}_{ii}\left(S_{k-1}^{(n)}\right)}\left(Z^{(n),i}_k\left(S^{(n)}_{k-1}\right)-\sum_{j<i}c^{(n)}_{ij}\left(S_{k-1}^{(n)}\right)\delta W^{(n),j}_k\right)&:\ c^{(n)}_{ii}\left(S_{k-1}^{(n)}\right)>0\\ \left(\Delta t^{(n)}\right)^{1/2}U_k^{(n),i}&:\ c^{(n)}_{ii}\left(S_{k-1}^{(n)}\right)=0\end{cases},
\end{equation}
then the following result is an immediate corollary of Lemma \ref{drei}. 

\begin{cor}\label{vier}
Let Assumptions \ref{scaling}, \ref{density} and \ref{U} be satisfied. Then for all $n,i\in\N$and $k=1,\dots,T_n$,
\[
	Z^{(n),i}_k\left(S^{(n)}_{k-1}\right) = \sum_{j \leq i} c^{(n)}_{ij}\left(S^{(n)}_{k-1}\right) \delta W^{(n),j}_{k}
\]
as well as
\[\E\left(\left.Z_k^{(n),i}\delta W^{(n),j}_k\right|\F_{k-1}^{(n)}\right)=\Delta t^{(n)} c_{ij}^{(n)}\left(S_{k-1}^{(n)}\right)
\quad\text{and}\quad
\E\left(\left.\delta W^{(n),i}_k\delta W^{(n),j}_k\right|\F_{k-1}^{(n)}\right)=\Delta t^{(n)}\delta_{ij}.\]
\end{cor}

In order to see that the random variables $\delta W^{(n),i}_k,\ i\in\N$, allow us to represent the volume process as a stochastic integral, we define for all $i,j,n\in\N$ a function $d_{ij}^{(n)}:E_{loc}\ra[-M,M]$ via
\[d_{ij}^{(n)}(s):=\begin{cases}\sigma_i^{(n)}(s)c_{ij}^{(n)}(s)&:\ j\leq i\\0&:\ j>i\end{cases}.\]

Note that for each $m\in\N$, the matrix $\left(d_{ij}^{(n)}(s)\right)_{i,j\leq m}$ is the triangular matrix that one obtains from the Cholesky factorization of the covariance matrix $\left(\sigma_i^{(n)}(s)\sigma_j^{(n)}(s)\rho_{ij}^{(n)}(s)\right)_{i,j\leq m}$. Therefore, the functions $\left(d^{(n)}_{ij}(s)\right)_{i,j\in\N}$ will serve as the volatility operator in the stochastic equation representing $V^{(n)}$. Indeed, equations (\ref{1moment}), (\ref{2moment}) and Corollary \ref{vier} imply that almost surely
\begin{equation}\label{SDE_V}
\begin{split}
V^{(n)}(t,x)&=V_0(x)+\sum_if_i(x)\sum_{k=1}^{\lfloor t/\Delta t^{(n)}\rfloor}\left\langle\delta V^{(n)}_k,f_i\right\rangle\\
&=V_0(x)+\sum_if_i(x)\sum_{k=1}^{\lfloor t/\Delta t^{(n)}\rfloor}\left[\mu_i^{(n)}\left(S^{(n)}_{k-1}\right)\Delta t^{(n)}+\sigma_i^{(n)}\left(S_{k-1}^{(n)}\right)\delta Z_k^{(n),i}\right]\\
&=V_0(x)+\sum_if_i(x)\sum_{k=1}^{\lfloor t/\Delta t^{(n)}\rfloor}\left[\mu_i^{(n)}\left(S^{(n)}_{k-1}\right)\Delta t^{(n)}+\sigma_i^{(n)}\left(S_{k-1}^{(n)}\right)\sum_{j\leq i}c^{(n)}_{ij}\left(S_{k-1}^{(n)}\right)\delta W_k^{(n),j}\right]
\end{split}
\end{equation}\normalsize
The convergence of the drift has already been established. In the following two subsections we prove the convergence of the volatility operator and the martingale driving the SDE.

\subsection{Convergence of the volatility operator}\label{vol-operator}

In this section we prove convergence of the functions $c^{(n)}_{ij}(\cdot)$ and $d^{(n)}_{ij}(\cdot)$. As a byproduct we obtain a key estimate for the functions $\alpha^{(n)}_{ij}(\cdot)$. This estimate allows, for instance, to verify that the random variables $\delta W^{(n),i}_k$, $k \in \N$, satisfy the Lindeberg condition in the proof of Theorem \ref {spacetime}.

\begin{lem}\label{ca}
Suppose that Assumptions \ref{scaling}, \ref{density}, \ref{gi}, and \ref{hi} are satisfied. Then there exist for every $i\in\N$ and $j\leq i$ functions $c_{ij},\alpha_{ij}:\ E_{loc}\ra\R$ such that
\[\sup_{s\in E_{loc}}\left|c_{ij}^{(n)}(s)-c_{ij}(s)\right|\ra0\quad\text{and}\quad\sup_{s\in E_{loc}}\left|\alpha_{ij}^{(n)}(s)-\alpha_{ij}(s)\right|\ra0.\]
Moreover, for all $i\in \N$ and $j\leq i$,
\[\inf_{s\in E_{loc}}c_{ii}(s)>0\quad\text{and}\quad \sup_{s\in E_{loc}}\left|\alpha_{ij}(s)\right|<\infty.\]
\end{lem}
\begin{proof}
The claim is proven by induction on $i$. Clearly, for $i=1$ we have $c_{11}\equiv 1\equiv\alpha_{11}$. Now assume the claim is true for all functions $c^{(n)}_{jl},\alpha_{jl}^{(n)}$ with $l\leq j\leq i-1$. Especially, this implies that for all $j<i$ and for $n$ large enough we have $\inf_{s\in E_{loc}}c_{jj}^{(n)}(s)>0$ and hence 
\[c_{ij}^{(n)}(s)=\frac{1}{c_{jj}^{(n)}(s)}\left(\rho^{(n)}_{ij}(s)-\sum_{l<j}c_{il}^{(n)}(s)c_{jl}^{(n)}(s)\right).\]
By iterative reasoning from $j=1$ to $j=i-1$ we see that this term converges uniformly in $s\in E_{loc}$ to some function $c_{ij}$ (defined via a similar recursion scheme) due to the induction hypothesis and Lemma \ref{rho1}. The same is then true for 
\[c_{ii}^{(n)}(s)=\left(1-\sum_{j<i}\left(c_{ij}^{(n)}(s)\right)^2\right)^{1/2}.\]
Next we have to show that the limit satisfies $\inf_{s\in E_{loc}}c_{ii}(s)>0$. First, note that by the induction hypothesis for large enough $n$, $c^{(n)}_{jj}(s)>0$ for all $j<i$ and hence by equation (\ref{WZ}), 
\begin{eqnarray*}
Z^{(n),i}(s)-\sum_{j<i}c_{ij}^{(n)}\left(s\right) W^{(n),j}(s)
&=&Z^{(n),i}(s)-\sum_{j<i}c_{ij}^{(n)}\left(s\right)\sum_{l\leq j}\alpha^{(n)}_{jl}\left(s\right)Z^{(n),l}(s)\\
&=&Z^{(n),i}(s)-\sum_{l<i}Z^{(n),l}(s)\sum_{l\leq j<i}c_{ij}^{(n)}\left(s\right)\alpha^{(n)}_{jl}\left(s\right).
\end{eqnarray*}
We set for all $l<i$,
\[\beta_l^{(n)}(s):=\begin{cases}\frac{-1}{\sigma^{(n)}_l(s)}\sum_{l\leq j<i}c_{ij}^{(n)}(s)\alpha^{(n)}_{jl}(s)&:\ \text{if }\sigma^{(n)}_l(s)>0\\0&:\ \text{else}\end{cases}\]
as well as
\[\beta_i^{(n)}(s):=\begin{cases}\frac{1}{\sigma^{(n)}_i(s)}&:\ \text{if }\sigma^{(n)}_i(s)>0\\0&:\ \text{else}\end{cases}.\]
By the induction hypothesis, Lemma \ref{sigpos}, and Lemma \ref{sigma2} we know that for every $j\leq i$ there exists a bounded function $\beta_j:\ E\ra\R$ such that
\begin{equation}\label{beta}
\sup_{s\in E_{loc}}\left|\beta^{(n)}_j(s)-\beta_j(s)\right|\ra0.
\end{equation}
But for $n$ large enough we have by definition for all $s\in E_{loc}$,
\begin{eqnarray*}
&& c_{ii}^{(n)}(s)W^{(n),i}(s) \\ &=&Z^{(n),i}(s)-\sum_{j<i}c_{ij}^{(n)}\left(s\right) W^{(n),j}(s)\\
&=&\Delta v^{(n)}\left\langle X_1^{(n)}(s)\sum_{j=0}^{\lfloor\cdot/\Delta x^{(n)}\rfloor}\1_{I^{(n)}\left(X_2^{(n)}(s)\right)}\left(x_j^{(n)}\right),\sum_{l\leq i}\beta_l^{(n)}\left(s\right)f_l\right\rangle-\Delta t^{(n)}\sum_{l\leq i}\beta_l^{(n)}\left(s\right)\mu_l^{(n)}(s)
\end{eqnarray*}
and then also
\begin{eqnarray*}
\Delta t^{(n)}\left(c_{ii}^{(n)}(s)\right)^2&=&\E\left(c^{(n)}_{ii}(s) W^{(n),i}(s)\right)^2\\
&=&\Delta t^{(n)}\int_0^\infty g^{(n)}(s;y)\left[\sum_{l\leq i}\beta^{(n)}_l(s)F^{(n)}_l(y)\right]^2dy-\left[\Delta t^{(n)}\sum_{l\leq i}\beta^{(n)}_l(s)\mu^{(n)}_l(s)\right]^2.
\end{eqnarray*}
Clearly, (\ref{beta}) implies that $\sup_{n\in\N}\sup_{s\in E_{loc}}\left|\beta_l^{(n)}(s)\right|=:C<\infty$ for all $l\leq i$. Hence, the last term on the right hand side in the above equation converges to zero uniformly in $s\in E_{loc}$ using that  $\sup_{n\in\N}\sup_{s\in E_{loc}}\left|\mu_l^{(n)}(s)\right|<\infty$ for all $l\leq i$ by Lemma \ref{mu}. Moreover,
\begin{eqnarray*}
& & \sup_{s\in E_{loc}}\left|\int_0^\infty \left(g^{(n)}(s;y)-g(s;y)\right)\left[\sum_{l\leq i}\beta^{(n)}_l(s)F^{(n)}_l(y)\right]^2dy\right|\\
& \leq & C^2i^2\cdot\sup_{s\in E_{loc}}\int_0^\infty \left|g^{(n)}(s;y)-g(s;y)\right|dy \ra 0
\end{eqnarray*}
and by dominated convergence we deduce that, uniformly in $s\in E_{loc}$,
\begin{eqnarray*}
\int_0^\infty g(s;y)\left[\sum_{l\leq i}\beta^{(n)}_l(s)F^{(n)}_l(y)\right]^2dy\ra
\int_0^\infty g(s;y)\left[\sum_{l\leq i}\beta_l(s)F_l(y)\right]^2dy.
\end{eqnarray*}
Therefore,
\begin{equation}\label{cii}
c_{ii}(s)=\int_0^\infty g(s;y)\left[\sum_{l\leq i}\beta_l(s)F_l(y)\right]^2dy.
\end{equation}
Now suppose that $\inf_{s\in E_{loc}}c_{ii}(s)=0$. Since $g(\cdot;y)$ is bounded away from zero for all $y\in\R_+$ by Assumption \ref{gi}, we deduce from (\ref{cii}) that there must exist an $E_{loc}$-valued sequence $(s_n)$ such that 
\[\sum_{l\leq i}\beta_l(s_n)F_l(y)\ra 0\quad\text{for almost all }y\in\R_+.\]
Since $\sup_{s\in E_{loc}}\left|\beta_l(s)\right|<\infty$ for all $l\leq i$, this implies that there exists some vector $b\in\R^i$ such that
\[\sum_{l\leq i}b_lF_l(y)=0\quad\text{for almost all }y\in\R_+\]
and thus also
\[H(y):=\sum_{l\leq i}b_lf_l(y)=0\quad\text{for almost all }y\in\R_+.\]
However,
\[0=\left\Vert H\right\Vert_{L^2}^2=\sum_{l\leq i}b_l^2\]
implies that $b_l=0$ for all $l\leq i$ and hence we must have $\beta_l(s_n)\ra 0$ for all $l\leq i$. But for $l=i$ this gives a contradiction, since
\[\sup_{s\in E_{loc}}\left(\sigma_i(s)\right)^2=\sup_{s\in E_{loc}}\int_0^\infty g(s;y)\left[F_i(y)\right]^2dy\leq M^2<\infty.\]
Hence, $\sigma_i$ is bounded and thus $\beta_i$ is bounded away from $0$. This proves that $\inf_{s\in E_{loc}}c_{ii}(s)>0$.\\
Now the convergence of the $\alpha_{ij}^{(n)},\ j\leq i,$ to some $\alpha_{ij}$ satisfying $\sup_{s\in E_{loc}}\left|\alpha_{ij}(s)\right|<\infty$ follows from the definition of the $\alpha_{ij}^{(n)}$ by backwards iteration from $j=i$ to $j=1$.
\end{proof}

The following remark is key for our subsequent analysis. 

\begin{rem}\label{q}
If Assumptions \ref{scaling}, \ref{density}, \ref{gi}, and \ref{hi} are satisfied, then there exists according to Lemmata \ref{sigpos}, \ref{sigma2}, and \ref{ca} for every $m\in\N$ a constant $q_m<\infty$ and an $n_m\in\N$ such that for all $n\geq n_m$ and $j\leq i\leq m$,
\[\sup_{s\in E_{loc}}\frac{\left|\alpha_{ij}^{(n)}(s)\right|}{\sigma^{(n)}_j(s)}<q_m.\]
\end{rem}

Let us snow turn to the convergence of the volatility operator. Similarly, to the functions $d^{(n)}_{ij}$ we set for all $i,j\in\N$ and $s\in E_{loc}$,
\[d_{ij}(s):=\begin{cases}\sigma_i(s)c_{ij}(s)&:\ j\leq i\\ 0&:\ j>i\end{cases}.\]

\begin{lem}\label{sumc} Given Assumptions \ref{scaling}, \ref{density}, \ref{gi}, and \ref{hi} we have for all $m\in\N$, 
\[\sup_{s\in E_{loc}}\sum_{i\in\mathcal{I}_m}\sum_{j\leq i}\left(d^{(n)}_{ij}(s)-d_{ij}(s)\right)^2\ra0.\]
\end{lem}
\begin{proof} Fix $m\in\N$ and let $\eps>0$. According to Lemma \ref{eps} we can a finite subset $J\subset\mathcal{I}_m$ such that for all $n\in\N$ and $y\in\R_+$,
\[\sum_{i\in\mathcal{I}_m\backslash J}\left(F^{(n)}_i(y)\right)^2\leq\frac{\eps}{8M^2}.\]
Hence for any $n\in\N$ and $s\in E_{loc}$,
\begin{eqnarray*}
&&\sum_{i\in\mathcal{I}_m\backslash J} \left(\sigma_i^{(n)}(s)\right)^2\sum_{j\leq i}\left(c_{ij}^{(n)}(s)-c_{ij}(s)\right)^2 \\ &\leq& 2\sum_{i\in\mathcal{I}_m\backslash J} \left(\sigma^{(n)}_i(s)\right)^2\sum_{j\leq i}\left[\left(c_{ij}^{(n)}(s)\right)^2+\left(c_{ij}(s)\right)^2\right]\\
&=& 4\sum_{i\in\mathcal{I}_m\backslash J} \left(\sigma_i^{(n)}(s)\right)^2 \\ &\leq& 4\sum_{i\in\mathcal{I}_m\backslash J}\int_0^\infty g^{(n)}(s;y)dx\left(F^{(n)}_i(y)\right)^2dy \leq \frac{\eps}{2}.
\end{eqnarray*}
According to Lemma \ref{ca} there exists for all $i,j\in\N$ an $n_{ij}=n_{ij}(\eps,m)$ such that for any $n\geq n_{ij}$,
\[\sup_{s\in E_{loc}}\left|c_{ij}^{(n)}(s)-c_{ij}(s)\right|^2<\frac{\eps}{2|J|M^2m}.\]
Hence, for any $n\geq n_0:=\max\{n_{ij}:\ j\leq i,\, i\in J\}$ and $s\in E_{loc}$,
\begin{eqnarray*}
\sum_{i\in\mathcal{I}_m}\left(\sigma_i^{(n)}(s)\right)^2\sum_{j\leq i}\left(c_{ij}^{(n)}(s)-c_{ij}(s)\right)^2&\leq&\frac{\eps}{2}+
\sum_{i\in J}\left(\sigma^{(n)}_i(s)\right)^2\sum_{j\leq i}\left(c_{ij}^{(n)}(s)-c_{ij}(s)\right)^2\\
&<&\frac{\eps}{2}+\sum_{i\in J}\left(\sigma^{(n)}_i(s)\right)^2\frac{\eps}{2M^2m}\leq\eps.
\end{eqnarray*}
Now the claim follows from the above and Lemma \ref{sigma2} because
\[\sum_{i\in\mathcal{I}_m}\sum_{j\leq i}\left(d^{(n)}_{ij}(s)-d_{ij}(s)\right)^2\leq 
2\sum_{i\in\mathcal{I}_m}\left(\sigma_i^{(n)}(s)\right)^2\sum_{j\leq i}\left(c^{(n)}_{ij}(s)-c_{ij}(s)\right)^2+2\sum_{i\in\mathcal{I}_m}\left(\sigma^{(n)}_i(s)-\sigma_i(s)\right)^2.\]
\end{proof}

\subsection{Convergence of the martingale to a Gaussian random measure}\label{gaussian}

We are now going to prove the convergence of the martingale driving the SDE in (\ref{SDE_V}) to a cylindrical Brownian motion on $L^2(\R_+)$. We start with the following simple lemma.

\begin{lem}\label{Cauchy}
Let Assumptions \ref{scaling}, \ref{density}, and \ref{U} be satisfied. Then there exists for any $\varphi\in L^2(\R_+)$ and $\eps>0$ an $m_0\in\N$ such that for all $m_2\geq m_1\geq m_0$, $n\in\N$, and $t\in[0,T]$,
\[\E\left(\sum_{k=1}^{\lfloor t/\Delta t^{(n)}\rfloor}\sum_{i=m_1+1}^{m_2}\delta W_k^{(n),i}\langle\varphi,f_i\rangle\right)^2<\eps.\]
\end{lem}

\begin{proof}
We choose
\[m_0:=\inf\left\{m\in\N:\ \sum_{i=m+1}^\infty\langle\varphi,f_i\rangle^2<\frac{\eps}{T}\right\}.\]
Then due to Corollary \ref{vier} we have for all $n\in\N$ and $t\in[0,T]$,
\[\E\left(\sum_{k=1}^{\lfloor t/\Delta t^{(n)}\rfloor}\sum_{i=m_1+1}^{m_2}\delta W_k^{(n),i}\langle\varphi,f_i\rangle\right)^2=
\sum_{k=1}^{\lfloor t/\Delta t^{(n)}\rfloor}\sum_{i=m_1+1}^{m_2}\Delta t^{(n)}\langle\varphi,f_i\rangle^2\leq T\sum_{i=m_1+1}^{m_2}\langle\varphi,f_i\rangle^2<\eps.\]
\end{proof}

The preceding lemma allows us to define for each $n\in \N$ a so called $L^2(\R_+)^\#$-semimartingale (for the definition see \cite{KP2}): for any $t\in[0,T]$ and $\varphi\in L^2(\R_+)$ we set
\begin{equation} \label{Wn}
	W^{(n)}(\varphi,t):=\sum_{k=1}^{\lfloor t/\Delta t^{(n)}\rfloor}\sum_i\delta W^{(n),i}_k \left\langle \varphi,f_i\right\rangle,
\end{equation}
where the above series is defined as the $L^2\left(\p^{(n)}\right)$-limit. 

\begin{thm}\label{spacetime}
Suppose that Assumptions \ref{scaling}, \ref{density}, \ref{gi}, \ref{hi}, and \ref{U} are satisfied.
Let $l\in\N$ and take any $\varphi_1,\dots,\varphi_l\in L^2(\R_+)$. Then as $n\ra\infty$, 
\[\left(W^{(n)}(\varphi_1,\cdot),\dots,W^{(n)}(\varphi_l,\cdot)\right)\RA\left(W(\varphi_1,\cdot),\dots,W(\varphi_l,\cdot)\right)\]
in $\mathcal{D}\left([0,T];\R^l\right)$, where $W$ is a cylindrical Brownian motion on $L^2(\R_+)$. Thus, in the terminology of \cite{KP2}, $W$ is a centered Gaussian $L^2(\R_+)^\#$-semimartingale with covariance structure 
\[
	\E\left[W(\varphi_1,t)W(\varphi_2,s)\right]=(t\wedge s)\langle\varphi_1,\varphi_2\rangle
\]	
for $\varphi_1,\varphi_2\in L^2(\R_+)$ and $s,t\in[0,T]$.
\end{thm}

\begin{proof}
For any $\varphi\in L^2(\R_+)$ we define the approximating sequence
\[\varphi^m:=\sum_{i=1}^m\langle\varphi,f_i\rangle f_i.\] 
Take $\varphi_1,\dots,\varphi_l\in L^2(\R_+)$ for some $l\in\N$. We will show that $\left(W^{(n)}(\varphi_1,\cdot),\dots,W^{(n)}(\varphi_l,\cdot)\right)$ converges to a centered Gaussian process with covariance function $$\E\left[W(\varphi_i,t)W(\varphi_j,s)\right]=(t\wedge s)\langle\varphi_i,\varphi_j\rangle$$ for any $1\leq i,j\leq l$ and $s,t\in[0,T]$. To this end, first note that for all $n\in\N$ and for all $k\leq T_n$, 
\begin{eqnarray*}
	\E\left(\left.W^{(n)}\left(\varphi_i,t_k^{(n)}\right)\right|\F^{(n)}_{k-1}\right) &=& \lim_{m\ra\infty}\E\left(\left.W^{(n)}\left(\varphi_i^m,t_k^{(n)}\right)\right|\F^{(n)}_{k-1}\right) \\
	&=& \lim_{m\ra\infty}W^{(n)}\left(\varphi_i^m,t_{k-1}^{(n)}\right) \\
	&=& W^{(n)}\left(\varphi,t_{k-1}^{(n)}\right).
\end{eqnarray*}
Secondly, for all $n\in\N$ and $k_1,k_2\in\{1,\dots, T_n\}$ denoting 
\[\delta W^{(n)}\left(\varphi_i,t_k^{(n)}\right):=W^{(n)}\left(\varphi_i,t_k^{(n)}\right)-W^{(n)}\left(\varphi_i,t_{k-1}^{(n)}\right),\]
we have
\begin{eqnarray*}
\E\left(\left.\delta W^{(n)}\left(\varphi_i,t_k^{(n)}\right)\delta W^{(n)}\left(\varphi_j,t_{k}^{(n)}\right)\right|\F^{(n)}_{k-1}\right)&=&
\lim_{m\ra\infty}\E\left(\left.\sum_{g,h=1}^m\delta W_k^{(n),g}\langle\varphi_i,f_g\rangle\delta W_k^{(n),h}\langle\varphi_j,f_h\rangle\right|\F^{(n)}_{k-1}\right)\\
&=&\lim_{m\ra\infty}\Delta t^{(n)}\sum_{h=1}^m\langle\varphi_i,f_h\rangle\langle\varphi_j,f_h\rangle=\Delta t^{(n)}\langle\varphi_i,\varphi_j\rangle
\end{eqnarray*}
and therefore for all $1\leq i,j\leq l$ and $t\in[0,T]$,
\[\sum_{k=1}^{\lfloor t/\Delta t^{(n)}\rfloor}\E\left(\left.\delta W^{(n)}\left(\varphi_i,t_k^{(n)}\right)\delta W^{(n)}\left(\varphi_j,t_k^{(n)}\right)\right|\F^{(n)}_{k-1}\right)\ra t\langle\varphi_i,\varphi_j\rangle\quad\text{a.s.}\]
In order to apply the functional convergence theorem for martingale difference arrays it remains to check that the conditional Lindeberg condition is satisfied. For ease of notation we will assume that $l=2$ in the following, noting that the general case follows by similar arguments. 

Let us fix some $\eps>0$ and $t\in[0,T]$. We want to show that for any $\delta>0$ there exists an $n_0=n_0(\eps,\delta)$ such that for all $n\geq n_0$,
\[\sum_{k=1}^{\lfloor t/\Delta t^{(n)}\rfloor}\E\left(\left.\left[\delta W^{(n)}\left(\varphi_1,t_k^{(n)}\right)\right]^2\1_{\left\{\left[ \delta W^{(n)}\left(\varphi_1,t_k^{(n)}\right)\right]^2+\left[\delta W^{(n)}\left(\varphi_2,t_{k}^{(n)}\right)\right]^2>\eps\right\}}\right|\F^{(n)}_{k-1}\right)<\delta\quad\text{a.s.}\]
To this end we first apply Lemma \ref{Cauchy} and choose $m=m(\delta)$ such that for all $n\in\N$,
\[\sum_{k=1}^{\lfloor t/\Delta t^{(n)}\rfloor}\E\left(\left.\left[\delta W^{(n)}\left(\varphi_1-\varphi_1^m,t_k^{(n)}\right)\right]^2\right|\F^{(n)}_{k-1}\right)=
\sum_{k=1}^{\lfloor t/\Delta t^{(n)}\rfloor}\Delta t^{(n)}\sum_{i=m+1}^\infty \langle\varphi_1,f_i\rangle^2<\frac{\delta}{4}.\]
Hence,
\begin{eqnarray*}
\sum_{k=1}^{\lfloor t/\Delta t^{(n)}\rfloor}\E\left(\left.\left[\delta W^{(n)}\left(\varphi_1,t_k^{(n)}\right)\right]^2\1_{\left\{\left[ \delta W^{(n)}\left(\varphi_1,t_k^{(n)}\right)\right]^2+\left[\delta W^{(n)}\left(\varphi_2,t_{k}^{(n)}\right)\right]^2>\eps\right\}}\right|\F^{(n)}_{k-1}\right)\qquad\qquad\qquad\\ 
<\frac{\delta}{2}+2\sum_{k=1}^{\lfloor t/\Delta t^{(n)}\rfloor}\E\left(\left.\left[\sum_{i=1}^m\delta W^{(n),i}_k\langle\varphi_1,f_i\rangle\right]^2\1_{\left\{\left[ \delta W^{(n)}\left(\varphi_1,t_k^{(n)}\right)\right]^2+\left[\delta W^{(n)}\left(\varphi_2,t_{k}^{(n)}\right)\right]^2>\eps\right\}}\right|\F^{(n)}_{k-1}\right).
\end{eqnarray*}
According to Remark \ref{q} there exists an $n_m\in\N$ and a constant $q_m<\infty$ such that for all $n\geq n_m$,
\begin{eqnarray*}
\sum_{i=1}^{m}\left(\delta W^{(n),i}_k\right)^2&=& 
\sum_{i=1}^{m}\left[\1_{\left\{c_{ii}^{(n)}\left(S_{k-1}^{(n)}\right)>0\right\}}\left(\sum_{j\leq i}\alpha^{(n)}_{ij}\left(S_{k-1}^{(n)}\right)Z_k^{(n),j}\right)^2+\1_{\left\{c_{ii}^{(n)}\left(S_{k-1}^{(n)}\right)=0\right\}}\Delta t^{(n)}\left(U^{(n),i}_k\right)^2\right]\\
&\leq& 
\sum_{i=1}^{m}\left[\1_{\left\{c_{ii}^{(n)}\left(S_{k-1}^{(n)}\right)>0\right\}}2^i\sum_{j\leq i}\left(\alpha^{(n)}_{ij}\left(S_{k-1}^{(n)}\right)Z_k^{(n),j}\right)^2+\1_{\left\{c_{ii}^{(n)}\left(S_{k-1}^{(n)}\right)=0\right\}}\Delta t^{(n)}\right]\\
&=&\sum_{i=1}^{m}\left[\1_{\left\{c_{ii}^{(n)}\left(S_{k-1}^{(n)}\right)>0\right\}}2^i\sum_{j\leq i}\left(\frac{\alpha^{(n)}_{ij}\left(S_{k-1}^{(n)}\right)}{\sigma^{(n)}_j\left(S_{k-1}^{(n)}\right)}\right)^2\left\langle\delta\ol{v}^{(n)}_k,f_j\right\rangle^2+\1_{\left\{c_{ii}^{(n)}\left(S_{k-1}^{(n)}\right)=0\right\}}\Delta t^{(n)}\right]\\
&\leq&\sum_{i=1}^{m}\left[2^mq_m^2\left\Vert\delta\ol{v}_k^{(n)}\1_{[0,m]}\right\Vert_{L^2}^2+\Delta t^{(n)}\right]
\stackrel{(\ref{vbound})}{\leq} \Delta t^{(n)}\left[m^2q_m^22^{m+1}M^2+m\right]\leq d^m_n\quad a.s.
\end{eqnarray*}
with $(d^m_n)_{n\in\N}$ being a deterministic sequence satisfying $d_n^m\ra0$ as $n\ra\infty$. We choose
\[n_0=n_0(\delta,\eps)=n_0(m(\delta),\delta,\eps):=\min\left\{n\in\N:\ 8T\left\Vert\varphi_1\right\Vert^2_{L^2}d_n^m\left(\left\Vert\varphi_1\right\Vert^2_{L^2}+\left\Vert\varphi_2\right\Vert^2_{L^2}\right)<\delta\eps\right\}.\]
Then for all $n\geq n_m$ by the Cauchy-Schwarz inequality,
\begin{eqnarray*}
&&\E\left(\left.\left[\sum_{i=1}^m\delta W^{(n),i}_k\langle\varphi_1,f_i\rangle\right]^2\1_{\left\{\left[ \delta W^{(n)}\left(\varphi_1,t_k^{(n)}\right)\right]^2+\left[\delta W^{(n)}\left(\varphi_2,t_{k}^{(n)}\right)\right]^2>\eps\right\}}\right|\F^{(n)}_{k-1}\right)\\
&\leq & \left\Vert\varphi_1\right\Vert^2_{L^2}\cdot
\E\left(\left.\sum_{i=1}^m\left(\delta W^{(n),i}_k\right)^2\left(\1_{\left\{\left[ \delta W^{(n)}\left(\varphi_1,t_k^{(n)}\right)\right]^2>\frac{\eps}{2}\right\}}+\1_{\left\{\left[\delta W^{(n)}\left(\varphi_2,t_{k}^{(n)}\right)\right]^2>\frac{\eps}{2}\right\}}\right)\right|\F^{(n)}_{k-1}\right)\\
&\leq & \frac{2\left\Vert\varphi_1\right\Vert^2_{L^2}d_n^m}{\eps}\cdot
\E\left(\left.\left[ \delta W^{(n)}\left(\varphi_1,t_k^{(n)}\right)\right]^2+\left[\delta W^{(n)}\left(\varphi_2,t_{k}^{(n)}\right)\right]^2\right|\F^{(n)}_{k-1}\right)\\
&=& \frac{2\left\Vert\varphi_1\right\Vert^2_{L^2}d_n^m}{\eps}\Delta t^{(n)}\left(\left\Vert\varphi_1\right\Vert^2_{L^2}+\left\Vert\varphi_2\right\Vert^2_{L^2}\right)<\frac{\delta\Delta t^{(n)}}{4T}\quad\text{a.s.}
\end{eqnarray*}
Hence, the conditional Lindeberg condition is satisfied and the functional central limit theorem for martingale difference arrays  (cf.~Theorem 3.33 in \cite{JS}) implies that
\[\left(W^{(n)}(\varphi_1,\cdot),\dots,W^{(n)}(\varphi_l,\cdot)\right)\RA\left(W(\varphi_1,\cdot),\dots,W(\varphi_l,\cdot)\right)\quad\text{in }\mathcal{D}([0,T];\R^l),\]
where $\left(W(\varphi_1,\cdot),\dots,W(\varphi_l,\cdot)\right)$ is a centered Gaussian process with covariance function \[\E\left[W(\varphi_i,t)W(\varphi_j,s)\right]=(t\wedge s)\langle\varphi_i,\varphi_j\rangle\] 
for any $1\leq i,j\leq l$ and $s,t\in[0,T]$.
\end{proof}

\begin{rem}
The process $W$ is not only an $L^2(\R_+)^\#$-semimartingale in the sense of \cite{KP2}, but can also be understood as a martingale random measure: If $\mathcal{A}:=\left\{A\subset\mathcal{B}(\R_+):\ A\ \text{bounded}\right\}$, we can define for any $A\in\mathcal{A}$ and $t\in [0,T]$, $M(A,t):=W(\1_A,t)$. Then $M$ is indeed a Gaussian martingale random measure indexed by $\mathcal{A}\times[0,T]$.
\end{rem}

\s{The state dynamics as an infinite dimensional SDE}\label{integrals}

In this section we show that the dynamics of $S^{(n)}$ can be written as an infinite dimensional SDE and prove the convergence of the integrands and integrators. Our concept of integration follows  \cite{KP2}, to which we refer for any  unknown terminology used in the following.

For each $n\in\N$ we define the $E_{loc}$-valued stochastic process $\left(S^{(n)}(t)\right)_{t\in[0,T]}$ as the piecewise constant interpolation of the $\left(S^{(n)}_k\right)_{k=0,\dots,T_n}$, i.e.
\[S^{(n)}(t):=S^{(n)}_k,\quad\text{if }t\in\left[t_{k}^{(n)},t_{k+1}^{(n)}\right).\]
Similarly, we set
\[B^{(n)}(t):=B^{(n)}_k,\qquad V^{(n)}(t,x):=V^{(n)}_k(x),\quad\text{if}\quad t_k^{(n)}\leq t<t^{(n)}_{k+1},\quad x\in\R_+.\]

In view of the equations (\ref{SDE_B}) and (\ref{SDE_V}) we have that
\begin{equation}\label{SDE_BV}
\begin{split}
	B^{(n)}(t) &= B_0^{(n)}+\sum_{k=1}^{\lfloor t/\Delta t^{(n)}\rfloor} \left[p^{(n)}\left(S^{(n)}_{k-1}\right)\Delta t^{(n)}+r^{(n)}\left(S^{(n)}_{k-1}\right)\delta Z^{(n)}_k\right] \\
	V^{(n)}(t,x) &=V^{(n)}_0(x)+\sum_if_i(x)\sum_{k=1}^{\lfloor t/\Delta t^{(n)}\rfloor}\left[\mu_i^{(n)}\left(S^{(n)}_{k-1}\right)\Delta t^{(n)}+\sigma_i^{(n)}\left(S_{k-1}^{(n)}\right)\sum_{j\leq i}c^{(n)}_{ij}\left(S_{k-1}^{(n)}\right)\delta W_k^{(n),j}\right].
\end{split}
\end{equation}

In terms of the processes $Z^{(n)}$ and $W^{(n)}$ introduced in (\ref{Zn}) and (\ref{Wn}), respectively, we can define a sequence of $L^2(\R_+)^\#$-semimartingales  $Y^{(n)}$ by putting, for any $n\in\N$, $t\in[0,T]$, and $\varphi\in L^2(\R_+)$, 
\[
	Y^{(n)}(\varphi,t):=\left(Z_k^{(n)},W^{(n)}(\varphi,t),t_k^{(n)}\right),\quad\text{if } t\in\left[t_k^{(n)},t_{k+1}^{(n)}\right).
\]
The stochastic integral with respect to $Y^{(n)}$ is introduced in Appendix \ref{appendix-integration}. If we define, for any $n \in \N$, the coefficient functions $G^{(n)}: E_{loc}\ra\hat{E}_{loc}$ (see Appendix \ref{appendix-integration} for the definition of the space $\hat{E}_{loc}$) via
\[G^{(n)}:=\left(G^{(n),1},\ 0,\ G^{(n),3},\ 0,\ G^{(n),5},\ G^{(n),6}\right)\]
with
\begin{align*}
&G^{(n),1}(s):=r^{(n)}\left(s \right),&G^{(n),5}(s;x,y)&:=\sum_{i}\sum_{j\leq i}d^{(n)}_{ij}\left(s \right)f_i(x)f_j(y),\\
&G^{(n),3}(s):=p^{(n)}\left(s \right),&G^{(n),6}(s;x)&:=\sum_i\mu^{(n)}_i\left(s \right)f_i(x)=\mu^{(n)}\left(s ;x\right),
\end{align*}
then the general integration theory guarantees that the integral 
\[
	\int_0^t G^{(n)}\left(S^{(n)}(u-)\right)dY^{(n)}(u),\quad t\in[0,T],
\]
is well-defined as an $E_{loc}$-valued stochastic process, and (\ref{SDE_BV}) yields the following representation of the state process:
\begin{equation}\label{representation}
	S^{(n)}(t)=S^{(n)}_0+\int_0^t G^{(n)}\left(S^{(n)}(u-)\right)dY^{(n)}(u),\quad t\in[0,T].
\end{equation}
In the next subsection we are going to prove the convergence of the integrators and integrands. 

\subs{Convergence of the integrator and integrand}

The following theorem shows that the sequence $Y^{(n)}$ converges to the $L^2(\R_+)^\#$-semimartingale 
\begin{equation}\label{Ydef}
Y(\varphi,t):=\left(Z(t),W(\varphi,t),t\right),\quad \varphi\in L^2(\R_+),\quad t\in[0,T],
\end{equation}
where $W$ is a cylindrical Brownian motion on $L^2(\R_+)$, and $Z$ is an independent standard Brownian motion. 

\begin{thm}\label{Y}
Let Assumptions \ref{prob1}, \ref{scaling}, \ref{density}, \ref{gi}, \ref{hi}, and \ref{U} be satisfied. Then, for every $k\in\N$ and $\varphi_1,\dots,\varphi_k\in L^2(\R_+)$,
\[\left(Y^{(n)}(\varphi_1,\cdot),\dots,Y^{(n)}(\varphi_k,\cdot)\right)\RA\left(Y(\varphi_1,\cdot),\dots,Y(\varphi_k,\cdot)\right)\]
in $\mathcal{D}\left([0,T];\R^{3k}\right)$, where $Y$ is defined in (\ref{Ydef}). 
\end{thm}

\begin{proof}
The joint convergence follows directly from Theorems \ref{Z} and \ref{spacetime} because the processes $Z^{(n)},\ n\in \N,$ and $W^{(n)}(\varphi,\cdot),\ n\in \N,$ are C-tight for any $\varphi\in L^2(\R)$. 
However, to derive the joint finite dimensional distributions (and especially to check the independence of the resulting cylindrical and standard Brownian motion), we have to show two more things: first, we will prove that for all $t\in[0,T]$ and $\varphi\in L^2(\R_+)$,
\[\sum_{k=1}^{\lfloor t/\Delta t^{(n)}\rfloor}\E\left(\left.\delta W^{(n)}\left(\varphi,t_k^{(n)}\right)\delta Z_k^{(n)}\right|\F_{k-1}^{(n)}\right)\ra0\quad\text{a.s.}\]
 and second, we will show that for all $\eps>0$, $t\in[0,T]$, and $\varphi\in L^2(\R_+)$, 
\[\sum_{k=1}^{\lfloor t/\Delta t^{(n)}\rfloor}\E\left(\left.\left(\left[\delta W^{(n)}\left(\varphi,t_k^{(n)}\right)\right]^2+\left[\delta Z_k^{(n)}\right]^2\right)\1_{\left\{\left[\delta W^{(n)}\left(\varphi,t_k^{(n)}\right)\right]^2+\left[\delta Z_k^{(n)}\right]^2>\eps\right\}}\right|\F_{k-1}^{(n)}\right)\ra0\quad\text{a.s.}\]
To this end, observe that for any $n,i\in\N$ and $k\leq T_n$,
\[\E\left(\left.\left\langle \delta \ol{v}_k^{(n)},f_i\right\rangle \delta Z_k^{(n)}\right| \F_{k-1}^{(n)}\right)=-\left(\Delta t^{(n)}\right)^2\mu^{(n)}_i\left(S_{k-1}^{(n)}\right)p^{(n)}\left(S_{k-1}^{(n)}\right).\]
Let $\delta>0$. We choose $m=m(\delta)$ such that for all $n\in\N$ and $t\in[0,T]$,
\begin{eqnarray*}
&&\sum_{k=1}^{\lfloor t/\Delta t^{(n)}\rfloor}\left|\E\left(\left.\sum_{i=m+1}^\infty\langle\varphi,f_i \rangle\delta W_k^{(n),i}\delta Z_k^{(n)}\right|\F_{k-1}^{(n)}\right)\right|\\
&\leq& \sum_{k=1}^{\lfloor t/\Delta t^{(n)}\rfloor}\left(\E\left[\left.\left(\delta Z^{(n)}_k\right)^2\right|\F^{(n)}_{k-1}\right]\E\left[\left.\left(\sum_{i=m+1}^\infty\langle\varphi,f_i\rangle\delta W^{(n),i}_k\right)^2\right|\F_{k-1}^{(n)}\right]\right)^{1/2}\\
&\leq& \sum_{k=1}^{\lfloor t/\Delta t^{(n)}\rfloor}\Delta t^{(n)}\left(\sum_{i=m+1}^\infty\langle\varphi,f_i\rangle^2\right)^{1/2}<\frac{\delta}{2}\quad\text{a.s.}
\end{eqnarray*}
Moreover for large enough $n$ and all $k\leq T_n$,
\begin{eqnarray*}
\E\left[\left.\sum_{i=1}^m\langle \varphi,f_i\rangle \delta W^{(n),i}_k\delta Z_k^{(n)}\right|\F^{(n)}_{k-1}\right]
&=&\sum_{i=1}^m\E\left[\left.\langle \varphi,f_i\rangle \sum_{j\leq i}\frac{\alpha_{ij}^{(n)}\left(S_{k-1}^{(n)}\right)}{\sigma^{(n)}_j\left(S_{k-1}^{(n)}\right)}\left\langle \delta \ol{v}^{(n)}_k,f_j\right\rangle\delta Z_k^{(n)}\right|\F^{(n)}_{k-1}\right]\\
&=&
-\left(\Delta t^{(n)}\right)^2p^{(n)}\left(S_{k-1}^{(n)}\right)\sum_{i=1}^m\langle \varphi,f_i\rangle\sum_{j\leq i}\frac{\alpha_{ij}^{(n)}\left(S_{k-1}^{(n)}\right)}{\sigma^{(n)}_j\left(S_{k-1}^{(n)}\right)} \mu^{(n)}_j\left(S_{k-1}^{(n)}\right).
\end{eqnarray*}
According to Lemma \ref{mu} and Remark \ref{q} there exist an $n_0=n_0(m)$ and a constant $C_m<\infty$ such that for all $n\geq n_0$,
\[\sup_{s\in E_{loc}}\left|\sum_{i=1}^m\langle \varphi,f_i\rangle \sum_{j\leq i}\mu^{(n)}_j\left(s\right)\frac{\alpha_{ij}^{(n)}(s)}{\sigma^{(n)}_j\left(s\right)}\right|
\leq C_m\sum_{i=1}^m\left|\langle \varphi,f_i\rangle\right|\leq mC_m\left\Vert\varphi\right\Vert_{L^2}<\infty.\]   
Hence for all $n\geq n_0$,
\begin{eqnarray*}
\sum_{k=1}^{\lfloor t/\Delta t^{(n)}\rfloor}\left|\E\left(\left.\sum_{i=1}^m\langle \varphi,f_i\rangle \delta W^{(n),i}_kZ_k^{(n)}\right|\F^{(n)}_{k-1}\right)\right|&\leq&\sum_{k=1}^{\lfloor t/\Delta t^{(n)}\rfloor}\left(\Delta t^{(n)}\right)^2\left|p^{(n)}\left(S_{k-1}^{(n)}\right)\right|mC_m\left\Vert\varphi\right\Vert_{L^2}\\
&\stackrel{(\ref{pbound})}{\leq}& T\Delta x^{(n)}mC_m\left\Vert\varphi\right\Vert_{L^2} <\frac{\delta}{2}\quad\text{a.s.}
\end{eqnarray*}
This proves that for any $\delta>0$ there exists $n_0=n_0(\delta)$ such that for all $n\geq n_0$,
\[\left|\sum_{k=1}^{\lfloor t/\Delta t^{(n)}\rfloor}\E\left(\left.\delta W^{(n)}\left(\varphi,t_k^{(n)}\right)\delta Z_k^{(n)}\right|\F_{k-1}^{(n)}\right)\right|<\delta \quad\text{a.s.}\]
Next, using the estimate in equation (\ref{Zbound}) we have almost surely
\begin{eqnarray*}
&&\sum_{k=1}^{\lfloor t/\Delta t^{(n)}\rfloor}\E\left(\left.\left[\delta Z_k^{(n)}\right]^2\1_{\left\{\left[\delta W^{(n)}\left(\varphi,t_k^{(n)}\right)\right]^2+\left[\delta Z_k^{(n)}\right]^2>\eps\right\}}\right|\F_{k-1}^{(n)}\right)\\
&\leq& c_n\sum_{k=1}^{\lfloor t/\Delta t^{(n)}\rfloor}\p\left(\left.\left[\delta W^{(n)}\left(\varphi,t_k^{(n)}\right)\right]^2>\frac{\eps}{2}\right|\F_{k-1}^{(n)}\right)+\p\left(\left.\left[\delta Z_k^{(n)}\right]^2>\frac{\eps}{2}\right|\F_{k-1}^{(n)}\right)\\
&\leq& \frac{2c_n}{\eps} \sum_{k=1}^{\lfloor t/\Delta t^{(n)}\rfloor}\E\left(\left.\left[\delta W^{(n)}\left(\varphi,t_k^{(n)}\right)\right]^2+\left[\delta Z_k^{(n)}\right]^2\right|\F_{k-1}^{(n)}\right)\leq \frac{2c_n}{\eps} \sum_{k=1}^{\lfloor t/\Delta t^{(n)}\rfloor}\Delta t^{(n)}\left(\left\Vert\varphi\right\Vert_{L^2}^2+1\right)\ra0.
\end{eqnarray*}
Furthermore,
\begin{eqnarray*}
&&\E\left(\left.\left[\delta W^{(n)}\left(\varphi,t^{(n)}_k\right)\right]^2\1_{\left\{\left[\delta W^{(n)}\left(\varphi,t_k^{(n)}\right)\right]^2+\left[\delta Z_k^{(n)}\right]^2>\eps\right\}}\right|\F_{k-1}^{(n)}\right)\\
&\leq & 2\cdot\E\left(\left.\left[\sum_{i=m+1}^\infty\delta W^{(n),i}_k\langle\varphi,f_i\rangle\right]^2+\left[\sum_{i=1}^m\delta W^{(n),i}_k\langle\varphi,f_i\rangle\right]^2\1_{\left\{\left[\delta W^{(n)}\left(\varphi,t_k^{(n)}\right)\right]^2+\left[\delta Z_k^{(n)}\right]^2>\eps\right\}}\right|\F_{k-1}^{(n)}\right)\\
&\leq & 2\Delta t^{(n)}\sum_{i=m+1}^\infty\langle\varphi,f_i\rangle^2+2\left\Vert\varphi\right\Vert_{L^2}^2\E\left(\left.\left[\sum_{i=1}^m\delta W^{(n),i}_k\right]^2\1_{\left\{\left[\delta W^{(n)}\left(\varphi,t_k^{(n)}\right)\right]^2+\left[\delta Z_k^{(n)}\right]^2>\eps\right\}}\right|\F_{k-1}^{(n)}\right)
\end{eqnarray*}
and by a similar reasoning as above
\[\E\left(\left.\left[\sum_{i=1}^m\delta W^{(n),i}_k\right]^2\1_{\left\{\left[\delta W^{(n)}\left(\varphi,t_k^{(n)}\right)\right]^2+\left[\delta Z_k^{(n)}\right]^2>\eps\right\}}\right|\F_{k-1}^{(n)}\right)
\leq \frac{2d_n^m}{\eps}\Delta t^{(n)}\left(\left\Vert\varphi\right\Vert^2_{L^2}+1\right).\]
Now for any $\delta >0$ we choose $m=m(\delta)$ and $n_0=n_0(m,\delta,\eps)=n_0(\delta,\eps)$ such that for all $n\geq n_0$,
\[\sum_{i=m+1}^\infty\langle\varphi,f_i\rangle^2<\frac{\delta}{4T}\qquad\text{and}\qquad\frac{2d_n^m}{\eps}\left\Vert\varphi\right\Vert_{L^2}^2\left(\left\Vert\varphi\right\Vert^2_{L^2}+1\right)<\frac{\delta}{4T}\]
and therefore
\[\sum_{k=1}^{\lfloor t/\Delta t^{(n)}\rfloor}\E\left(\left.\left[\delta W^{(n)}\left(\varphi,t^{(n)}_k\right)\right]^2\1_{\left\{\left[\delta W^{(n)}\left(\varphi,t_k^{(n)}\right)\right]^2+\left[\delta Z_k^{(n)}\right]^2>\eps\right\}}\right|\F_{k-1}^{(n)}\right)<\delta\quad\text{a.s.}\]
\end{proof}

Let us now turn to the integrands. The results of Section \ref{volume} suggest that the coefficient functions $G^{(n)}$ converge in a local sense  to 
\[G=\left(G^{1},\ 0,\ G^{3},\ 0,\ G^{5},\ G^{6}\right): E_{loc}\ra\hat{E}_{loc}\]
with
\begin{align*}
&G^{1}(s):=r\left(s\right),&G^{5}(s;x,y)&:=\sum_i\sum_{j\leq i}d_{ij}\left(s \right)f_i(x)f_j(y),\\
&G^{3}(s):=p\left(s\right),&G^{6}(s;x)&:=\sum_i\mu_i\left(s\right)f_i(x)=\mu\left(s ;x\right).
\end{align*}

In order to formulate the convergence result we define for every $m\in\N$ the projections of $G^5$ and $G^6$ on $[0,m]$ as
\[G^{5,m}(s;x,y):=\sum_{i\in\mathcal{I}_m}\sum_{j\leq i}d_{ij}\left(s\right)f_i(x)f_j(y),\qquad G^{6,m}(s;x):=\sum_{i\in\mathcal{I}_m}\mu_i\left(s\right)f_i(x)=\mu\left(s;x\right),
\]
and set 
\[G^m(s):=\left(G^{1}(s),\ 0,\ G^{3}(s),\ 0,\ G^{5,m}(s),\ G^{6,m}(s)\right),\quad  s\in E_{loc}.\]
Moreover, for all $m\in\N$ we define the space
\[E_{m}:=\left\{s=\left(b,v\1_{[0,m]}\right):\ (b,v)\in E_{loc}\right\}\subset E_{loc}.
\] 
Next, we approximate $G^{(n)}$ by 
functions $G^{(n)}_{m}: E_m\ra\hat{E},\ m\in\N$, given by
\[{G}^{(n)}_{m}:=\left(G^{(n),1}_m,\ 0,\ G^{(n),3}_m,\ 0,\ {G}^{(n),5}_{m},\ {G}^{(n),6}_{m}\right),\]
where for all $s\in E_m\subset E_{loc}$ and $s_m := (s \wedge m, v)$,
\begin{align*}
&G^{(n),1}_m(s):=p^{(n)}\left(s_m\right),& G^{(n),5}_{m}(s;x,y)&:=\sum_{i\in\mathcal{I}_{m}}\sum_{j\leq i}d^{(n)}_{ij}\left(s\right)f_i(x)f_j(y),\\
&G^{(n),3}_m(s):=r^{(n)}\left(s_m\right),& G^{(n),6}_{m}(s;x,y)&:=\sum_{i\in\mathcal{I}_{m}}\mu_i^{(n)}\left(s\right)f_i(x).
\end{align*}

Analogously, we define for each $m\in\N$ a function ${G}_m: E_m\ra\hat{E}$ via a similar modification of $G$, i.e.~we have
\[G_m(s):=G^m\left(s_m\right),\quad s\in E_m.\]
We note that for $s=(b,v)\in E_m$ with $b\leq m$, $G_m(s)=G^m(s)$, due to Assumptions \ref{prob22}, \ref{gii}, and \ref{hii}. 

\begin{thm}\label{Fconv}
Let Assumptions \ref{scaling}, \ref{prob2}, \ref{density}, \ref{g}, and \ref{h} hold. Then for any $m\in\N$,
\[\sup_{s\in E_{m}}\left\Vert G^{(n)}_{m}(s)-G_{m}(s)\right\Vert_{\hat{E}}\ra0.\]
\end{thm}

\begin{proof}
By Assumption \ref{prob2}, Lemma \ref{mu}, and Lemma \ref{sumc} we have for all \mbox{$s=(b,v)\in E_{m}$,}
\begin{align*}
&\left\Vert G^{(n)}_{m}(s)-G_{m}(s)\right\Vert_{\hat{E}}=\left|r^{(n)}\left(s_m\right)-r\left(s_m\right)\right|+\left|p^{(n)}\left(s_m\right)-p\left(s_m\right)\right|\\&\qquad\qquad\qquad\qquad\qquad+\left(\sum_{i\in\mathcal{I}_{m}}\left(\mu_i^{(n)}(s)-\mu_i\left(s\right)\right)^2\right)^{1/2}+\left(\sum_{i\in\mathcal{I}_{m}}\sum_{j\leq i}\left(d^{(n)}_{ij}\left(s\right)-d_{ij}\left(s\right)\right)^2\right)^{1/2}.
\end{align*}
\end{proof}

\subsection{Compactness of the integrands}

In this section it is shown that for each $m\in\N$ the $G^{(n)}_m,\ n\in\N,$ satisfy a uniform compactness condition from which we shall later deduce relative compactness of the price-volume process and hence the existence of accumulation points. 

\begin{thm}\label{compact3}
Given Assumptions \ref{scaling}, \ref{density}, and \ref{hi}, there exists for every $m\in\N$ a compact set $K_{m}\subset \hat{E}$ such that for all $n\in\N$ and $s\in E_{m}$,
\[G^{(n)}_{m}(s)\in K_{m}.\]
\end{thm}

Since $G^{(n),1}_{m}$ and $G^{(n),3}_{m}$ are uniformly bounded by Assumption \ref{prob21}, we only have to care about the last two components of $G^{(n)}_{m}$. Thus, Theorem \ref{compact3} will directly follow from Lemmata \ref{compact2} and \ref{compact1} below.

\begin{lem}\label{compact2}
Let Assumptions \ref{scaling} and \ref{density} be satisfied. Then for each $m\in\N$ the set
\[K^{5}_{m}:=\left\{G^{(n),5}_{m}(s):\ s\in E_{m},\ n\in \N\right\}\subset L^2\left(\R_+^2\right)\] t
 is relatively compact.
\end{lem}

\begin{proof}
First note that for all $s,\tilde{s}\in E_{m}$ we have
\[\left\Vert G^{(n),5}_{m}(s)-G^{(n),5}_{m}\left(\tilde{s}\right)\right\Vert^2_{L^2\left(\R^2_+\right)}=\sum_{i\in\mathcal{I}_{m}}\sum_{j\leq i}\left(d^{(n)}_{ij}\left(s\right)-d^{(n)}_{ij}\left(\tilde{s}\right)\right)^2.\]
Now consider a sequence $\left(G^{(n_k),5}_{m}(s_k)\right)_{k\in\N}\subset K^{5}_{m}$ and set $a_k:=2^{-k}, \ k\in\N$. W.l.o.g.~we may assume that $s_k\in E_m$ for all $k\in\N$. As in the proof of Lemma \ref{sumc} one can show that there exists a finite index set $J\subset\mathcal{I}_{m}$ such that for all $n\in\N$ and $s\in E_{loc}$,
\[\sum_{i\in \mathcal{I}_{m}\backslash J}\left(\sigma_i^{(n)}\left(s\right)\right)^2<\frac{a_1}{8}.\]
For any $(i,j)\in\N^2$, $\left(d^{(n_k)}_{ij}\left(s_k\right)\right)_{k\in\N}$ is a real-valued sequence, bounded by $M$. Since $J$ is a finite set, there exists a subsequence $(k_q)\subset\N$ and a $q_0=q_0(a_1)\in\N$ such that for each pair $(i,j)$ with $i\in J$ and $j\leq i$,
\[\left(d^{(n_{k_q})}_{ij}\left(s_{k_q}\right)-d^{(n_{k_{q'}})}_{ij}\left(s_{k_{q'}}\right)\right)^2\leq\frac{a_1}{|J|(|J|+1)}\quad\text{for all }q,q'\geq q_0.\]
Hence, for all $q,q'\geq q_0$ we have
\begin{align*}
\sum_{i\in\mathcal{I}_{m}}\sum_{j\leq i}\left(d^{(n_{k_q})}_{ij}\left(s_{k_q}\right)-d^{(n_{k_{q'}})}_{ij}\left(s_{k_{q'}}\right)\right)^2
&\leq \frac{a_1}{2} + 
2\sum_{i\in\mathcal{I}_{m_0}\backslash J}\sum_{j\leq i}\left\{\left(d^{(n_{k_q})}_{ij}\left(s_{k_q}\right)\right)^2+\left(d^{(n_{k_{q'}})}_{ij}\left(s_{k_{q'}}\right)\right)^2\right\}\\
&= \frac{a_1}{2}+2\sum_{i\in\mathcal{I}_{m}\backslash J}\left\{\left(\sigma^{(n_{k_q})}_{i}\left(s_{k_q}\right)\right)^2+\left(\sigma^{(n_{k_{q'}})}_{i}\left(s_{k_{q'}}\right)\right)^2\right\} 
< a_1.
\end{align*}
Next, we consider the sequence $\left(G^{(n_{k_q}),5}_{m}(s_{k_q})\right)_{q\in\N}\subset K^{5}_{m}$ and construct in a similar way as above - with $a_1$ being replaced by $a_2$ - a further subsequence. This will be done iteratively for all $a_k,\ k\in\N$. Finally, we choose the diagonal sequence of all these subsequences, which will be a Cauchy sequence and hence convergent in $L^2(\R^2_+)$. This shows that $K^{5}_{m}$ is relatively compact.
\end{proof}

\begin{lem}\label{compact1}
Let Assumption \ref{hi} be satisfied. Then for each $m\in\N$ the set
\[K^{6}_{m}:=\left\{G^{(n),6}_{m}(s):\ s\in E_{m},\ n\in \N\right\}\subset L^2(\R_+)\]
 is relatively compact. 
\end{lem}

\begin{proof}
Consider some sequence $\left(G^{(n_k),6}_{m}(s_k)\right)_{k\in\N}\subset K^{6}_{m}$ and set again $a_k:=2^{-k},\ k\in\N$. As in the proof of Lemma \ref{compact2} we may assume that $s_k\in E_m$ for all $k\in\N$. By Assumption \ref{hi} there exists $K>0$ such that for all $n\in\N$ and $s\in E_{loc}$,
\[\int_0^\infty \left|h^{(n)}(s;y)\right|^2dy<K.\]
We apply Lemma \ref{eps} to find a finite subset $J\subset\mathcal{I}_{m}$ such that for all $n\in\N$ and $y\in\R_+$,
\[\sum_{i\in\mathcal{I}_m\backslash J}\left(F^{(n)}_i(y)\right)^2\leq\frac{a_1}{8Km}.\]
Hence for all $n\in\N$ and $s\in E_{loc}$,
\begin{eqnarray*}
\sum_{i\in\mathcal{I}_{m}\backslash J}\left(\mu_i^{(n)}(s)\right)^2&=&
\sum_{i\in\mathcal{I}_{m}\backslash J}\left(\int_0^{\infty} h^{(n)}(s;y)F_i^{(n)}(y)dy\right)^2 \\
&\leq& \sum_{i\in\mathcal{I}_{m}\backslash J}m\int_0^{m} \left(h^{(n)}(s;y)F_i^{(n)}(y)\right)^2dy\\
&\leq& \frac{a_1}{8K}\int_0^\infty\left|h^{(n)}(s;y)\right|^2dy < \frac{a_1}{8}.
\end{eqnarray*}
The rest of the proof follows as in the proof of Lemma \ref{compact2}.
\end{proof}

\subsection{Continuity of the integrand}\label{continuity}

In this subsection we will prove for all $m\in\N$ the continuity of $G_{m}$.
First note that by Assumption \ref{prob2} there exists some $L>0$ such that for all $s=(b,v),\ \tilde{s}=\left(\tilde{b},\tilde{v}\right)\in E_{m}$,
\[\left|G^1_{m}(s)-G^1_{m}(\tilde{s})\right|\leq L\left(1+|b| +|\tilde{b}| \right)\left(1+\left\Vert v\right\Vert_{L^2}+\left\Vert \tilde{v}\right\Vert_{L^2}\right)\left(\left|b-\tilde{b}\right|+\left\Vert v-\tilde{v}\right\Vert_{L^2}\right).
\]
Hence, for any $c>0$ there exists $L_c<\infty$ such that for all $s,\tilde{s}\in E_m$ with $\left\Vert s\right\Vert_{E}\leq c ,\ \left\Vert\tilde{s}\right\Vert_{E}\leq c$,
\[\left|G^1_{m}(s)-G^1_{m}(\tilde{s})\right|\leq L_c\left\Vert s-\tilde{s}\right\Vert_{E}.\]
A similar result holds for $G^3_{m}$. It remains to show the continuity of $G_{m}^5$ and $G_{m}^6$.

\begin{lem}\label{c4}
Under Assumption \ref{h} there exists for every $m\in\N$ and $c>0$ a constant $L^m_c$ such that for all $s,\tilde{s}\in E_{m}$ with $\left\Vert s\right\Vert_{E}\leq c,\ \left\Vert\tilde{s}\right\Vert_{E}\leq c$ we have
\[\left\Vert G_{m}^6(s)-G_{m}^6(\tilde{s})\right\Vert_{L^2(\R_+)}\leq L_c^m\left\Vert s-\tilde{s}\right\Vert_{E}.\]
\end{lem}

\begin{proof}
Due to Assumption \ref{h} we have for all $s,\tilde{s}\in E_{m}$,
\begin{eqnarray*}
\left\Vert G_{m}^6(s)-G_{m}^6(\tilde{s})\right\Vert_{L^2}^2&=&\sum_{i\in\mathcal{I}_{m}}\left(\mu_i(s)-\mu_i(\tilde{s})\right)^2 \\
&=& \sum_{i\in\mathcal{I}_{m}}\left(\int_0^\infty \left[h\left(s;y\right)-h\left(\tilde{s};y\right)\right]F_i(y)dy\right)^2\\
&\leq&\sum_{i\in\mathcal{I}_{m}}m\int_0^m \left[h(s;y)-h(\tilde{s};y)\right]^2\left(F_i(y)\right)^2dy \\
&\leq& m^2\int_0^\infty \left[h(s;y)-h(\tilde{s};y)\right]^2dy\\
&\leq& m^2L^2\left(1+|b|+|\tilde{b}|\right)^2\left(1+\left\Vert v\right\Vert_{L^2}+\left\Vert \tilde{v}\right\Vert_{L^2}\right)^2\left(\left|b-\tilde{b}\right|+\left\Vert v-\tilde{v}\right\Vert_{L^2}\right)^2.
\end{eqnarray*}
\end{proof}

\begin{lem}\label{dLip}
Suppose that Assumptions \ref{scaling}, \ref{density}, \ref{g}, and \ref{h} are satisfied. Then there exists for all $c>0$ and $m,i,j\in\N$ with $j\leq i$ a constant $L^{m,c}_{ij}>0$ such that for all $s,\tilde{s}\in E_{m}$ with \mbox{$\left\Vert s\right\Vert_{E}\leq c,\ \left\Vert\tilde{s}\right\Vert_{E}\leq c$,}
\[|d_{ij}(s)-d_{ij}(\tilde{s})|\leq L^{m,c}_{ij}\left\Vert s-\tilde{s}\right\Vert_{E}.\]
\end{lem}

\begin{proof}
Since $d_{ij}=\sigma_i c_{ij}$ for all $j\leq i$ and $|\sigma_i|\leq M,\ |c_{ij}|\leq 1$, it is sufficient to show the inequality for $\sigma_i$ and $c_{ij}$ separately. For all $i,j\in\N$ and $s,\tilde{s}\in E_{m}$ by Assumption \ref{g},
\begin{eqnarray*}
\left|\sigma_i(s)\sigma_j(s)\rho_{ij}(s)-\sigma_i(\tilde{s})\sigma_j(\tilde{s})\rho_{ij}(\tilde{s})\right|
&\leq&\int_0^\infty |g(s;y)-g(\tilde{s};y)||F_i(y)F_j(y)|dy\\
&\leq& L\left(1+|b|+|\tilde{b}|\right)\left(1+\left\Vert v\right\Vert_{L^2}+\left\Vert \tilde{v}\right\Vert_{L^2}\right)\left(\left|b-\tilde{b}\right|+\left\Vert v-\tilde{v}\right\Vert_{L^2}\right).
\end{eqnarray*}
In the case $i=j$, using the fact that $\inf_{s\in E_{loc}}\sigma_i(s)>0$ by Lemma \ref{sigpos}, we can thus find $L^{m,c}_i>0$ for each $i\in\N$ such that for all $s,\tilde{s}\in E_{m}$ with $\left\Vert s\right\Vert_E\leq c,\ \left\Vert\tilde{s}\right\Vert_E\leq c$,
\[\left|\sigma_i(s)-\sigma_i(\tilde{s})\right|=\frac{\left|\sigma^2_i(s)-\sigma_i^2(\tilde{s})\right|}{\sigma_i(s)+\sigma_i(\tilde{s})}\leq L^{m,c}_i\left\Vert s-\tilde{s}\right\Vert_E.\]
Using again the boundedness away from zero of $\sigma_i$ and $\sigma_j$, we may also find $K^{m,c}_{ij}>0$ for each $(i,j)$ such that  for all $s,\tilde{s}\in E_{m}$ with $\left\Vert s\right\Vert_E\leq c,\ \left\Vert\tilde{s}\right\Vert_E\leq c$,
\begin{eqnarray*}
|\rho_{ij}(s)-\rho_{ij}(\tilde{s})|&\leq& \frac{\left|\sigma_i(s)\sigma_j(s)\rho_{ij}(s)-\sigma_i(\tilde{s})\sigma_j(\tilde{s})\rho_{ij}(\tilde{s})\right|+\left|\sigma_i(s)\sigma_j(s)-\sigma_i(\tilde{s})\sigma_j(\tilde{s}))\right|}{\sigma_i(s)\sigma_j(s)}\\
&\leq& K^{m,c}_{ij}\left\Vert s-\tilde{s}\right\Vert_E.
\end{eqnarray*}
Because of the recursive definition of the $c_{ij},\ j\leq i,$ as functions of the $\rho_{ij},\ j\leq i,$ the same inequality  (with a different constant) follows for each $c_{ij}$ from the fact that all the $c_{ij},\ j\leq i,$ are bounded by $1$ and $\inf_{s\in E_{loc}} c_{ii}(s)>0$ for all $i\in\N$ by Lemma \ref{ca}.
\end{proof}

\begin{lem}\label{c3}
Let Assumptions \ref{scaling}, \ref{density}, \ref{g}, and \ref{h} be satisfied. If $(s_n)\subset D\left(E_{m};[0,T]\right)$ is a sequence with $\sup_{u\leq t}\left\Vert s_n(u)-s(u)\right\Vert_{E}\ra0$ for $t\in[0,T]$, then also
\[\sup_{u\leq t}\left\Vert G_{m}^5\left(s_n(u)\right)-G_{m}^5\left(s(u)\right)\right\Vert_{L^2(\R_+^2)}\ra0.\]
\end{lem}

\begin{proof}
Fix $\eps>0$ and let $(s_n)\subset D\left(E_{m};[0,T]\right)$ be any sequence satisfying $\sup_{u\leq t}\left\Vert s_n(u)-s(u)\right\Vert_{E}\ra0$. Then there exists $c>0$ such that $\left\Vert s(u)\right\Vert_{E}\leq c$ and $\left\Vert s_n(u)\right\Vert_{E}\leq c$ for all $n\in\N$ and $u\in[0,t]$. Similarly to the proof of Lemma \ref{sumc} we can find a finite index set $J\subset\mathcal{I}_{m}$  such that for all $\tilde{s}\in E_{loc}$,
\[\sum_{i\in\mathcal{I}_{m}\backslash J}\sigma_i^2\left(\tilde{s}\right)< \frac{\eps}{8}.\]
Moreover, by Lemma \ref{dLip} we can find an $n_0=n_0(\eps,c)$ such that for all $n\geq n_0$ and $u\leq t$,
\[\left(d_{ij}(s_n(u))-d_{ij}(s(u))\right)^2\leq\frac{\eps}{|J|(|J|+1)}\quad\forall\ i\in J,\ j\leq i.\]
Thus for all $n\geq n_0$,
\begin{eqnarray*}
\sup_{u\leq t}\left\Vert G^3_{m}(s_n(u))-G_{m}^3(s(u))\right\Vert^2_{L^2\left(\R_+^2\right)}&=&\sup_{u\leq t}\sum_{i\in\mathcal{I}_{m}}\sum_{j\leq i}\left(d_{ij}\left(s_n(u)\right)-d_{ij}\left(s(u)\right)\right)^2\\
&\leq& \frac{\eps}{2} + 2\sup_{u\leq t}\sum_{i\in\mathcal{I}_{m}\backslash J}\sum_{j\leq i}\left\{\left(d_{ij}\left(s_n(u)\right)\right)^2+\left(d_{ij}\left(s(u)\right)\right)^2\right\}\\
&=&\frac{\eps}{2}+2\sup_{u\leq t}\sum_{i\in\mathcal{I}_{m}\backslash J}\left\{\left(\sigma_{i}\left(s_n(u)\right)\right)^2+\left(\sigma_{i}\left(s(u)\right)\right)^2\right\}<\eps.
\end{eqnarray*}
\end{proof}

The preceding results immediately yield the following theorem.

\begin{thm}\label{cts}
Let Assumptions \ref{scaling}, \ref{prob2}, \ref{density}, \ref{g}, and \ref{h} be satisfied. 
If $(s_n)\subset D\left(E_{m};[0,T]\right)$ is a sequence such that $\sup_{u\leq t}\left\Vert s_n(u)-s(u)\right\Vert_{E}\ra0$ for $t\in[0,T]$, then also
\[\sup_{u\leq t}\left\Vert G_{m}\left(s_n(u)\right)-G_{m}\left(s(u)\right)\right\Vert_{\hat{E}}\ra0\qquad\forall\ m\in\N.\]
\end{thm}

\s{Convergence of the stochastic integrals}\label{result}

Before stating our main result, we need one more assumption on the convergence of the initial values.  

\begin{ass}\label{initial}
There exists $S_0=(B_0,V_0)\in E_{loc}$ such that for all $m\in\N$,
\[\left|B^{(n)}_0-B_0\right|+\left\Vert \left(V^{(n)}_0-V_0\right)\1_{[0,m]}\right\Vert_{L^2}\ra0.\]
\end{ass}

For all $n,m\in\N$ we set 
\[S^{(n),m}_0:=\left(B_0^{(n)},V^{(n)}_0\1_{[0,m]}\right),\quad S^{m}_0:=\left(B_0,V_0\1_{[0,m]}\right)\] 
and denote by $\tilde{S}^{(n),m}$ the solution of 
\[\tilde{S}^{(n),m}(t)=S^{(n),m}_0+\int_0^t G^{(n)}_{m}\left(\tilde{S}^{(n),m}(u-)\right)dY^{(n)}(u),\quad t\in[0,T].\]
Furthermore, we define for all $m,n\in\N$ the stopping time
\begin{eqnarray*}
\tau^{(n)}_{m}:=\inf\left\{t\geq0:\ B^{(n)}(t)\geq m\right\}\wedge T
\end{eqnarray*}
and the process
\[S^{(n),m}(t):=\left(B^{(n)}\left(t\wedge\tau^{(n)}_{m}\right),V^{(n)}\left(t\wedge\tau^{(n)}_{m}\right)\1_{[0,m]}\right),\quad t\in[0,T].\]
Note that, due to Assumptions \ref{prob22}, \ref{gii}, and \ref{hii} for all $n,m\in\N$ the process $\tilde{S}^{(n),m}$ equals $S^{(n),m}$ on $\left[0,\tau^{(n)}_{m}\right]$ and
\[\tau^{(n)}_{m}=\inf\left\{t\geq0:\ \tilde{B}^{(n),m}(t)\geq m\right\}\wedge T\quad\text{a.s.}\]

\begin{df}
We say that $S$ is a (global) solution of the infinite dimensional SDE
\begin{equation}\label{ISDE}
S(t)=S_0+\int_0^tG(S(u))dY(u),\quad t\in[0,T],
\end{equation}
if there exists a filtration $(\F_t)$ to which $S=(B,V)$ and $Y$ are adapted and for {\bf all} $m\in\N$,
\[\left(B(t),V(t)\1_{[0,m]}\right)=S_0^m+\int_0^tG^m(S(u))dY(u),\quad t\in[0,T].\]
We say that $(S,\tau,m)$ is a local solution of (\ref{ISDE})
if there exists a filtration $(\F_t)$ to which $S=(B,V)$ and $Y$ are adapted, $\tau$ is an $(\F_t)$-stopping time, and $S=(B,V)$ satisfies the SDE
\[\left(B(t\wedge\tau),V(t\wedge\tau)\1_{[0,m]}\right)=S_0^m+\int_0^{t\wedge\tau}G^m(S(u))dY(u),\quad t\in[0,T].\]
\end{df}

\subs{Local relative compactness of the state process}

Our main result states that the sequence of LOB models is relatively compact after localization and that any accumulation point is the solution to a certain infinite dimensional SDE driven by a pair consisting of a Brownian motion and a cylindrical Brownian motion.  

\begin{thm}\label{main}
Under Assumptions \ref{prob1}, \ref{scaling}, \ref{prob2}, \ref{density}, \ref{g}, \ref{h}, \ref{U}, and \ref{initial} the sequence $\left(S^{(n),m}\right)_{n\in\N}$ is relatively compact for all $m\in\N$ 
and any limit point $S^{m}=(B^{m},V^{m})$ gives a local solution $\left(S^{m},\tau_{m},m\right)$ of (\ref{ISDE}), i.e.~for $(t,x)\in[0,T]\times[0,m]$,
\begin{equation}
\begin{split}
B^{m}(t\wedge\tau_{m})&=B_0^{m}+\int_0^{t\wedge\tau_{m}}p\left(S^{m}(u)\right)du+\int_0^{t\wedge\tau_{m}}r\left(S^{m}(u)\right)dZ(u),\\
V^{m}(t\wedge\tau_{m},x)&=V_0^{m}+\int_0^{t\wedge\tau_{m}}\mu\left(S^{m}(u);x\right)du+\sum_{i\in\mathcal{I}_{m}}f_i(x)\sum_{j\leq i}\int_0^{t\wedge\tau_{m}}d_{ij}\left(S^{m}(u)\right)dW^j(u),
\end{split}
\end{equation}
where $W^j,\ j\in\N,$ and $Z$ are independent Brownian motions and $\tau_{m}:=\inf\left\{t\geq0:\ B^m(t)\geq m\right\}\wedge T$.
\end{thm}

For the proof we will apply Theorem 7.6 of \cite{KP2} and also partially follow the idea of the proof of Theorem 5.4 in \cite{KP1}. However, note that there is a crucial difference between our Theorem \ref{main} and Theorem 5.4 in \cite{KP1}: while in \cite{KP1} a local convergence result is derived by stopping the process appropriately and thereby localizing it in time, we do not only localize in time, but in fact have to localize in space as well. 

\begin{proof}
Let us fix $m\in\N$. First, we will show that the sequence $\left(S^{(n),m}_0,\tilde{S}^{(n),m},Y^{(n)}\right)_{n\in\N}$ is relatively compact. To do this we will apply Theorem 7.6 in \cite{KP2}. Let us verify the conditions of Theorem 7.6 in \cite{KP2}: Corollary \ref{Y} and Theorem \ref{ut} show that $\left(Y^{(n)}\right)_{n\in\N}$ is uniformly tight and converges weakly to $Y$ in terms of finite dimensional distributions. Moreover, by Assumption \ref{initial} there exists $S^{m}_0\in E_{m}$ such that $S_0^{(n),m}\ra S_0^m$. Hence, $\left(S^{(n),m}_0,Y^{(n)}\right)\RA(S_0^{m},Y)$. Theorems \ref{Fconv} and \ref{cts} imply that $G^{(n)}_{m},\ n\in\N,$ and $G_{m}$ satisfy
Condition C.2 of \cite{KP2}. Moreover, the compactness condition follows from Theorem \ref{compact3} and we clearly have $\sup_n\sup_{s\in E_{m}}\left\Vert G^{(n)}_{m}(s)\right\Vert_{\hat{E}}<\infty$, due to Assumption \ref{prob21}, Lemma \ref{mu}, and Lemma \ref{sigma2}. Hence, the requirements of Theorem 7.6 in \cite{KP2} are satisfied and we may conclude that the sequence $\left(S_0^{(n),m},\tilde{S}^{(n),m},Y^{(n)}\right)_{n\in\N}$ is relatively compact.

Next note that $\tau_{m}^{(n)}$ is a measurable function of $\tilde{S}^{(n),m}$ for all $n\in\N$, say $\tau_{m}^{(n)}=h_{m}\left(\tilde{S}^{(n),m}\right)$. We denote by $D_{h_{m}}$ the set of discontinuities of $h_{m}$. Then equation (\ref{pricenonconstant}) of Assumption \ref{prob1} ensures that \mbox{$\p(S^{m}\in D_{h_{m}})=0$} for any limit point $S^m$ of $\tilde{S}^{(n),m}$ and 
 we may conclude by the continuous mapping theorem that the sequence 
$\left(S^{(n),m}_0,\tilde{S}^{(n),m}\left(\cdot\wedge\tau^{(n)}_{m}\right),\tau^{(n)}_{m},Y^{(n)}\right)_{n\in\N}$ is also relatively compact. Let $\left({S}^{m}_0,\hat{S}^{m},\tau_{m}^0,Y\right)$ denote a weak limit point of that sequence. 
Then Condition C.2 together with Theorem 5.5 in \cite{KP2} yields that along a subsequence,
\[S^{(n),m}_0+\int_0^\cdot G^{(n)}_{m}\left(\tilde{S}^{(n),m}\left(u\wedge\tau^{(n)}_{m}\right)\right)dY^{(n)}(u)\ \RA\  S_0^{m}+\int_0^\cdot G_{m}\left(\hat{S}^{m}(u)\right)dY(u).\]
Furthermore, as remarked earlier $\tilde{S}^{(n),m}$ and $S^{(n),m}$ agree on $\left[0,\tau^{(n)}_{m}\right]$. Thus, by definition
\[S^{(n),m}\left(t\right)=\tilde{S}^{(n),m}\left(t\wedge\tau_m^{(n)}\right)=S^{(n),m}_0+\int_0^{t\wedge\tau^{(n)}_{m}} G^{(n)}_{m}\left(\tilde{S}^{(n),m}\left(u\right)\right)dY^{(n)}(u),\quad t\in[0,T].\]
Since $\hat{\tau}_{m}:=h_{m}\left(\hat{S}^{m}\right)\leq \tau_{m}^0$ a.s.~and since $G_m\left(\hat{S}_m(u)\right)=G^m\left(\hat{S}_m(u)\right)$ for $u\leq\hat{\tau}_m$, we conclude that $\left(S^{(n),m}\right)_{n\in\N}$ is relatively compact and that any limit point $\hat{S}^{m}$ of $\left(S^{(n),m}\right)_{n\in\N}$ gives a local solution of (\ref{ISDE}).
\end{proof}

\subs{Local weak convergence}

So far we have shown that the sequence of our LOB model dynamics is relatively compact in a localized sense and that any accumulation point solves a certain infinite dimensional SDE. If the limiting SDE admits a unique strong solution, then the LOB dynamics converges to a unique limit as shown by the following theorem.  

\begin{thm}\label{main4}
Suppose that all the assumptions of Theorem \ref{main} are satisfied and that for all $m\in\N$ there exists a unique strong solution $\hat{S}^m=(\hat{B}^m,\hat{V}^m)$ of
\begin{equation}\label{SDEmstopped}
\hat{S}^m(t)=S_0^m+\int_0^{t\wedge\tau_{m,m}}G_m\left(\hat{S}^m(u)\right)dY(u),\quad t\in[0,T],\quad\tau_{m,l}:=\inf\{t\geq0: \hat{B}^m(t)\geq l\}\wedge T.
\end{equation}
Then there exists a unique global solution $S=(B,V)$ of (\ref{ISDE}) and for all $m\in\N$, 
\[S^{(n),m}\RA S^{m}\quad\text{in}\quad\mathcal{D}\left([0,T]; E\right),\]
where $S^{m}(t):=\left(B\left(t\wedge\tau_{m}\right),V\left(t\wedge\tau_{m}\right)\1_{[0,m]}\right),\ t\in[0,T],$ and $\tau_{m}:=\inf\left\{t\geq0:\ B(t)\geq m\right\}\wedge T$.
\end{thm}

\begin{proof}
Strong uniqueness implies together with Assumptions \ref{prob2}, \ref{g}, and \ref{h} that for all $m,k\in\N$, $\hat{S}^m$ equals $\left(\hat{B}^{m+k},\hat{V}^{m+k}\1_{[0,m]}\right)$ almost surely on the interval $\left[0,\tau_{m,m}\wedge\tau_{m+k,m}\right]$. Thus for all $m,k\in\N$, $\tau_{m,m}=\tau_{m+k,m}$ and hence $\tau_{m,m}\leq \tau_{m+k,m+k}$ a.s.
Setting $\tau_0^0:=0$ we define
\[B(t):=\sum_{m=1}^\infty\1_{\left[\tau_{m-1,m-1},\tau_{m,m}\right)}(t)\hat{B}^m(t),\quad t\in[0,T],\]
and for all $m\in\N$ and $x\in[m-1,m)$,
\[V(t,x):=\1_{\left[0,\tau_{m,m}\right)}(t)\hat{V}^{m}(t,x)+\sum_{k=1}^\infty\1_{\left[\tau_{m+k-1,m+k-1},\tau_{m+k,m+k}\right)}(t)\hat{V}^{m+k}(t,x),\quad t\in[0,T].\]
Let $\tau_m:=\inf\{t\geq0:\ B(t)\geq m\}$ and $\tau_\infty:=\lim\tau_m$. Then by the linear growth condition of Assumption \ref{prob21} we have for all $m\in\N$ and $t\in[0,T]$,
\begin{eqnarray*}
\E\left[ B\left(t\wedge\tau_{m}\right)\right]^2&\leq& 4\left[B_0^2+\E\left(\int_0^t\left|p\left(S(u\wedge\tau_{m})\right)\right|du\right)^2+\E\left(\int_0^tr\left(S(u\wedge\tau_m)\right)dZ_u\right)^2\right]\\
&\leq& 4\left[B_0^2+T\int_0^t\E\left|p\left(S(u\wedge\tau_{m})\right)\right|^2du+\int_0^t\E\left|r\left(S(u\wedge\tau_m)\right)\right|^2du\right]\\
&\leq& 4B_0^2+2(4T+1)K^2\int_0^t\left(1+\E\left[ B\left(u\wedge\tau_{m}\right)\right]^2\right)du,
\end{eqnarray*}
which implies by Gronwall's inequality that
\[\E\left[ B\left(T\wedge\tau_\infty\right)\right]^2\leq\liminf_m\E\left[ B\left(T\wedge\tau_{m}\right)\right]^2<\infty.\]
Therefore $\tau_\infty=T$ a.s.~and, since $G^m\left(S(u)\right)=G_m\left(S(u)\right)$ on $\{u\leq \tau_{m}\}$ for all $m\in\N$, $S:=(B,V)$ defines a global solution of $(\ref{ISDE})$, which must be unique as well.
Now the weak convergence result follows from Theorem \ref{main}.
\end{proof}

\subsubsection{Uniqueness}
We are now going to analyse two classes of models which fit in the framework developed so far and which satisfy the assumptions of Theorem \ref{main4}, i.e.~they converge - in a local sense -  in the scaling limit to the unique solution of the infinite dimensional SDE (\ref{ISDE}). For this it is sufficient to establish the local Lipschitz continuity of the coefficient function $G_{m}$ on $E_m$ for all $m\in\N$, so that (\ref{SDEmstopped}) will have a unique solution for all $m\in\N$. Note that $G^1_{m},\ G^3_{m}$, and $G^6_{m}$ are locally Lipschitz continuous and uniformly bounded on $E_{m}$ by Assumption \ref{prob2} and Lemmata \ref{c4} and \ref{mu}. Hence, it remains to establish the local Lipschitz continuity of $G^5_{m}$.

\begin{lem}\label{main3}
Suppose in addition to the assumptions of Theorem \ref{main} that $g(s;y)$ is independent of the state of the book for all $y\in\R_+$. Then each $G_m$ is locally Lipschitz continuous and the conditions of Theorem \ref{main4} are satisfied. 
\end{lem}

\begin{proof}
Since in this case $G_m^5$ does not depend on $s$, $G_m^5$ is trivially Lipschitz continuous and uniformly bounded on $E_m^*$. Therefore there exists a unique strong solution of
\begin{equation}\label{SDEm}
\tilde{S}^m(t)=S_0^m+\int_0^tG_m\left(\tilde{S}^m(u)\right)dY(u),\quad t\in[0,T],
\end{equation}
by Corollary 7.8 in \cite{KP2}. Hence, there exists a unique strong solution of (\ref{SDEmstopped}) as well.
\end{proof}

The next lemma allows the volatility of the cumulated volume process to be state dependent. However, we require the dynamics of the system to only depend on current volumes through some approximation of the cumulated volume function. 

For all $l_0,m\in\N$ we define the index sets $\mathcal{I}_m(l_0):=\left\{i\in\mathcal{I}_m:\ l(i)<l_0\right\}$ and $\mathcal{I}(l_0):=\left\{i\in\N:\ l(i)<l_0\right\}$.

\begin{ass}\label{class}
There is $l_0\in\N$ such that for all pairs $s=(b,v)$, $\tilde{s}=(\tilde{b},\tilde{v})\in E_{loc}$ satisfying $b=\tilde{b}$ and $\langle v,f_i\rangle=\langle\tilde{v},f_i\rangle\  \forall\ i\in \mathcal{I}(l_0)$, we have the equalities
$p^{(n)}(s)=p^{(n)}(\tilde{s}),\ q^{(n)}(s)=q^{(n)}(\tilde{s})$, $h^{(n)}(s;y)=h^{(n)}(\tilde{s};y),\ g^{(n)}(s;y)=g^{(n)}(\tilde{s};y)$ for all $n\in\N,\  y\in\R_+$.
\end{ass}

\begin{lem}\label{main2}
Let the assumptions of Theorem \ref{main} and Assumption \ref{class} be satisfied. Then there exists a unique strong solution of $(\ref{SDEmstopped})$ for each $m\in\N$. 
\end{lem}

\begin{proof}
Fix $m\in\N$. We first show that there exists a unique strong solution to (\ref{SDEm}). Note that $p$ and $r$ are Lipschitz continuous on $E_m$ by Assumption \ref{prob2} and it follows from Lemmata \ref{c4} and \ref{dLip} that $d_{ij}$ and $\mu_i$ are also locally Lipschitz continuous on $E_m$. Moreover, each $\mu_i$ resp.~$d_{ij}$ is uniformly bounded and $p$ and $r$ satisfy a linear growth condition by Assumption \ref{prob21}. Therefore, the {\it finite} dimensional SDE 
\begin{eqnarray*}
\ol{B}^m(t)&=&B_0+\int_0^tp\left(\ol{S}^m(u)\right)du+\int_0^tr\left(\ol{S}^m(u)\right)dZ_u,\\
\ol{V}^m_i(t)&=&\langle V^m_0,f_i\rangle+\int_0^t\mu_i\left(\ol{S}^m(u)\right)du+\sum_{j\leq i}\int_0^td_{ij}\left(\ol{S}^m(u)\right)dW^j_u,\quad i\in \mathcal{I}_m(l_0),\\
\text{with}\quad \ol{V}^m&:=&\sum_{i\in\mathcal{I}_m(l_0)}\ol{V}^m_if_i\quad\text{and}\quad \ol{S}^m=\left(\ol{B}^m,\ol{V}^m\right)
\end{eqnarray*}
has a unique strong solution $\ol{S}^m$. Given this solution let us define
\[\ul{V}^m(t):=V^m_0+\sum_{i\in\mathcal{I}_m}f_i\int_0^t\mu_i\left(\ol{S}^m(u)\right)du+\sum_{i\in\mathcal{I}_m}f_i\sum_{j\leq i}\int_0^td_{ij}\left(\ol{S}^m(u)\right)dW^j_u,\quad t\in [0,T].\]
Clearly, $\left(\ol{B}^m,\ul{V}^m\right)$ is a solution of (\ref{SDEm}) due to Assumption \ref{class} and by construction it must be unique. It follows that $S^m:=\left(\ol{B}^m(\cdot\wedge\tau_m),\ul{V}^m(\cdot\wedge\tau_m)\right)$ is the unique strong solution of (\ref{SDEmstopped}). 
\end{proof}

\begin{rem}\label{final}
It is not clear to us how to establish the (strong or weak) uniqueness of a solution to the general infinite dimensional SDE of Theorem \ref{main} apart from the two cases considered in this section. Indeed, even though the Cholesky factorization in finite dimensions is a Lipschitz continuous operation, it is known that the Lipschitz constant grows dramatically when the dimension is increased. This makes the search for conditions on $g$ and $h$ that yield strong uniqueness very difficult. Of course, one could alternatively look for weak uniqeness of a solution to the infinite dimensional SDE by considering the associated martingale problem. However, to the best of our knowledge also in this case the problem is still unsolved and requires further research.
\end{rem}

\subsubsection{Examples}

We close this section with two examples where uniqueness of solutions to the limiting SDE can indeed be established.

\begin{ex}\label{ex1}
Let $\alpha,K,\eta,q>0$ and suppose that there exist $c_1,c_2\in(0,\infty)$ such that for all $n\in\N$ and $s=(b,v)\in E_{loc}$ with $0\leq b\leq c_1$ and $\left\Vert v\1_{[0,c_1]}\right\Vert_\infty<c_2$, the functions $p^{(n)}$ and $r^{(n)}$ are given by
\begin{eqnarray*}
p^{(n)}(s)&=&b\int^{b}_{(b-q)^+}\left(\alpha y-v(y)\right)dy+\eta\\
\left(r^{(n)}(s)\right)^2&=&\Delta x^{(n)}\eta+b^2.
\end{eqnarray*}
This specifies uniquely the conditional distribution of the process $B^{(n)}$ (as long as $S^{(n)}$ does not exit the $c_1$-$c_2$-interval defined above). We have chosen $r^{(n)}$ and $p^{(n)}$ such that the volatility as well as the absolute value of the drift of the price process are increasing in the price itself. Moreover, high volumes at the top of the book (compared to some reference level specified by $\alpha$) lead to a negative drift for the price process, while low volumes at the top of the book lead to a positive drift. In the scaling limit the price follows the volume-dependent, ``generalized Black-Scholes'' dynamics
\[dB(t)=\left(B(t)\int^{B(t)}_{(B(t)-q)^+}(\alpha y-v(y))dy+\eta\right)dt+B(t)dZ(t).\]
Order placements / cancelations outside the spread are assumed to be of unit size, i.e.~\mbox{$\p\left(\omega_k^{(n)}=\pm1\right)=1$} for all $n\in\N,\ k\leq T_n$. Furthermore, we suppose that there exist two functions $f^{(n)}_\pm: E_{loc}\times\R_+\ra\R_+$ for every $n\in\N$ such that for all $B\in\mathcal{B}(\R_+)$ and $k=1,\dots,T_n$,
\[\p\left(\left.\phi_k^{(n)}=C,\ \omega_k^{(n)}=\pm1,\ \pi_k^{(n)}\in B\ \right|\F_{k-1}^{(n)}\right)=\int_B f^{(n)}_\pm\left(S_{k-1}^{(n)};y\right)dy\quad \text{a.s.}\]
Let $h:\R\ra\R_+$ be continuously differentiable with bounded derivative and suppose that $h$ has compact support in $\R_-$. Let $D>0$ and suppose that the $f^{(n)}_\pm$ are for all $y\in\R_+$ and $s=(b,v)\in E_{loc}$ with $0\leq b\leq c_1$ and $\left\Vert v\1_{[0,c_1]}\right\Vert_\infty<c_2$ given by
\small
\begin{align*}
&f^{(n)}_+(s;y)=\left(1-\Delta p^{(n)}\left(r^{(n)}(s)\right)^2\right)\frac{\exp(-y)}{\left(2+\Delta v^{(n)}\right)}\left(1-\Delta v^{(n)}\left\langle v(\cdot+b)\1_{[-b,0]},h\right\rangle+\frac{\Delta v^{(n)}}{1+|y-b|}\right),\\
&f^{(n)}_-(s;y)=\left(1-\Delta p^{(n)}\left(r^{(n)}(s)\right)^2\right)\frac{\exp(-y)}{\left(2+\Delta v^{(n)}\right)}\left(1+\Delta v^{(n)}\left\langle v(\cdot+b)\1_{[-b,0]},h\right\rangle+\frac{\Delta v^{(n)}|y-b|}{1+|y-b|}\right).
\end{align*}\normalsize
This means that the location at which order placements and cancelations take place is exponentially distributed. Order cancelations are more likely to happen further away from the current best bid price or if cumulated volumes are quite high. On the other hand, order placements occur more frequently in the proximity of the current best bid price or if cumulated volumes are low. The above specification of $f^{(n)}$ yields for $s$ as above
\[g(s;y)=\frac{1}{2}\exp\left(-y\right)\]
and
\[h(s;y)=\frac{1}{2}\exp\left(-y\right)\left(-2\left\langle v(\cdot+b)\1_{[-b,0]},h\right\rangle+\frac{1-|y-b|}{1+|y-b|}\right).\]
Therefore, the covariance structure does not depend on $s$, which implies that $d_{ij}(s)=d_{ij}(\tilde{s})$ for all $s,\tilde{s}$ as above and $i,j\in\N$. However, note that $h$ and hence also $\mu$ depend on $s$. 
Moreover, one can check that for $n$ large enough all assumptions of Theorem \ref{main} are satisfied. Hence, Theorem \ref{main4} and Lemma \ref{main2} imply that the limiting SDE has a unique solution in this case and that $S^{(n)}$ converges weakly to this solution in a localized sense.
\end{ex}

While the above example shows that even with constant $G^5$ we can already model many interesting dependencies, one disadvantage is that the conditional distribution of the location variables $\pi_k^{(n)},\ n\in\N,\ k\leq T_n,$ of order placements resp.~cancelations cannot be taken to be relative to the current best bid price, which would be reasonable from a microeconomic point of view. Another disadvantage is that for constant $G^5$ the $L^2(\R_+)$-valued process $V$ is not necessarily positive respectively increasing in $x\in\R_+$. Nevertheless, for short time horizons Example \ref{ex1} can be viewed as a reasonable model of the bid side of a limit order book. 

The next example allows to model the location of order placements being distributed relative to the current best bid price.

\begin{ex}\label{ex2} 
For given $s=(b,v)\in E_{loc}$ we set
\[v_{l_0}(y):=\sum_{i\in\mathcal{I}(l_0)}\langle v,f_i\rangle f_i(y),\quad y\in\R_+.\]
Note that $v_{l_0}$ is the projection of $v$ on the subspace spanned by $\{f_i: i\in\mathcal{I}(l_0)\}$, which consists of all step functions on the grid $k2^{-l_0},\ k\in\N$. Hence, $v_{l_0}$ has the alternative representation
\[v_{l_0}(y)=\sum_{k\in\N_0} a_k\1_{[k2^{-l_0},(k+1)2^{-l_0})}(y)\quad\text{with}\quad a_k:=2^{-l_0}\int_{k2^{-l_0}}^{(k+1)2^{-l_0}}v(x)dx.\]
Therefore, $\left\{v_{l_0}(y): y\leq 2^{-l_0}\lfloor b2^{l_0}\rfloor\right\}$ only depends on $\left\{v(y): y\leq b\right\}$ for any $s=(b,v)\in E_{loc}$. Similarly to Example \ref{ex1} we specify the price dynamics as follows : let $\alpha,K,\eta,c_1,c_2>0$, $q\geq 2^{-l_0}$ and suppose that for all $n\in\N$ and $s=(b,v)\in E_{loc}$ with $0\leq b\leq c_1$ and $\left\Vert v_{l_0}\1_{[0,c_1]}\right\Vert_\infty<c_2$,
\begin{eqnarray*}
p^{(n)}(s)&=&b\int^{\left\lfloor b2^{l_0}\right\rfloor 2^{-l_0}}_{(b-q)^+}\left(\alpha y-v_{l_0}(y)\right)dy+\eta\\
\left(r^{(n)}(s)\right)^2&=&\Delta x^{(n)}\eta+b^2.
\end{eqnarray*}
As in Example \ref{ex1} we suppose that there exists a function $f^{(n)}$ for every $n\in\N$  such that for all $A\in\mathcal{B}([-M,M])$, $B\in\mathcal{B}(\R_+)$, and $k=1,\dots,T_n$,
\[\p\left(\left.\phi_k^{(n)}=C,\ \omega_k^{(n)}\in A,\ \pi_k^{(n)}\in B\ \right|\F_{k-1}^{(n)}\right)=\int_B\int_A f^{(n)}\left(S_{k-1}^{(n)};x,y\right)dxdy\quad \text{a.s.}\]
For $s=(b,v)\in E_{loc}$ with $0\leq b\leq c_1$ and $\left\Vert v_{l_0}\1_{[0,c_1]}\right\Vert_\infty<c_2$ let
\[f^{(n)}(s;x,y):=C_n(s)f^{(n),1}(v;x,y)\exp\left(-\frac{1}{2}(y-b)^2\right),\]
where $C_n(s)$ is chosen such that $\int_0^\infty\int_{-M}^Mf^{(n)}(s;x,y)dxdy=1-\Delta p^{(n)}\left(r^{(n)}(s)\right)^2$. It can be shown that as $n\ra\infty$, $C_n(s)$ converges to a function $C(b)$ depending on $b$ only. 
The function $f^{(n),1}$ specifies the conditional distribution of the order size and is given by
\[f^{(n),1}(v;x):=\frac{1-a_n(v)}{M}\1_{[0,M]}(x)+\frac{a_n(v)}{M}\1_{\left[-M,0\right]}(x)\]
with 
\[a_n(v):=\frac{1}{2}-\Delta v^{(n)}\left\langle v_{l_0}\left(\cdot+\lfloor b2^{l_0}\rfloor 2^{-l_0}\right)\1_{\left[-\lfloor b2^{l_0}\rfloor 2^{-l_0},0\right]},h\right\rangle,\]
where $h$ is as in Example \ref{ex0}. In this case
\[g(s;y)=\frac{M^2}{3}C(b)\exp\left(-\frac{1}{2}(y-b)^2\right)\]
as well as
\[h(s;y)=MC(b)\exp\left(-\frac{1}{2}(y-b)^2\right)\left\langle v_{l_0}\left(\cdot+\lfloor b2^{l_0}\rfloor 2^{-l_0}\right)\1_{[-\lfloor b2^{l_0}\rfloor 2^{-l_0},0]},h\right\rangle\]
both depend on $s$. Hence, also the $d_{ij},\ i,j\in\N,$ will vary with $s$. Still, it can be easily checked that all assumptions of Lemma \ref{main2} and Theorem \ref{main4} are satisfied.
\end{ex}

\appendix

\renewcommand{\theequation}{\thesection.\arabic{equation}}
\setcounter{equation}{0}

\section{Orthogonal decomposition of sequences of random variables}\label{decomposition}

In this appendix we derive an orthogonal decomposition result for sequences of random variables. Specifically, on a probability space $(\Omega, {\mathcal F}, \mathbb{P})$ we consider a sequence of normalised random variables $Z^i$ and denote by $\rho_{ij}$ the correlation between the variables $Z^i$ and $Z^j$ $(i,j \in \N, j \leq i)$. In terms of these quantities we define an array of real numbers $c_{ij},\ j\leq i,$ as well as a sequence of random variables $ W^{i},\ i\in\N$, via the following algorithm:
\begin{align*}
&\text{Put }c_{11}:=1\quad\text{and}\quad W^{1}:=Z^{1}.\\
&\text{For }i=2,3,4,\dots:\\
&\qquad\text{For }j=1,2,\dots,i-1:\\
&\qquad\qquad\text{If }c_{jj}=0,\\
&\qquad\qquad\quad\text{then } c_{ij}:=0.\\
&\qquad\qquad\text{Else }\\
&\qquad\qquad\quad c_{ij}:=\frac{1}{c_{jj}}\left(\rho_{ij}-\sum_{l<j}c_{il}c_{jl}\right).\\
&\qquad\text{Next $j$}.\\
&\qquad c_{ii}:=\left(1-\sum_{j<i}\left(c_{ij}\right)^2\right)^{1/2}\\
&\qquad  W^{i}:=\begin{cases}\frac{1}{c_{ii}}\left(Z^{i}-\sum_{j<i}c_{ij} W^{j}\right)&:\ c_{ii}\>0\\U^{i}&:\ c_{ii}=0\end{cases}.\\
&\text{Next }i.
\end{align*}

\begin{lem}\label{drei} For all $n,i\in\N$, $j\leq i$, the following holds:
\begin{align*}
\text{1.  }\ \E\left(Z^{i}W^{j}\right)= c_{ij},\qquad
\text{2.  }\ \sum_{j<i}c_{ij}^2\leq 1,\qquad
\text{3.  }\ \E\left( W^{i} W^{j}\right)= \delta_{ij},\qquad
\text{4.  }\ Z^{i}=\sum_{j\leq i}c_{ij}\ W^{j}.
\end{align*}
\end{lem}

\begin{proof}
We proceed by induction over $i$. For $i=1$ we have $c_{11}\equiv 1$ and $ W^{1}=Z^{1}$, which trivially gives 1.-4.

Now assume that 1.-4.~are true up to index $i-1$; in particular $c_{jj}$ and $ W^{j}$ are well defined for $j<i$. We first show 1.~for indices $i$ and $j<i$. 
This will be done by induction over $j$. For $j=1$, we have by definition
\[\E\left(Z^{i} W^{1}\right)=\E\left(Z^{i}Z^{1}\right)=\rho_{i1}=c_{i1}.\]
Now consider an arbitrary $j<i$ and suppose that the claim is true up to $j-1$. If $c_{jj}=0$, then 
\[\E\left(Z^{i}W^{j}\right)=\E\left(Z^{i}U^{j}\right)=0=c_{ij}.\]
If $c_{jj}>0$, then by definition and the induction hypothesis
\begin{eqnarray*}
\E\left(Z^{i} W^{j}\right)=
\frac{1}{c_{jj}}\E\left(Z^{i}\left(Z^{j}-\sum_{l\leq j-1}c_{jl} W^{l}\right)\right)
= \frac{1}{c_{jj}}\left(\rho_{ij}-\sum_{l\leq j-1}c_{il}c_{jl}\right)
=c_{ij}.
\end{eqnarray*}
This implies that
\begin{eqnarray*}
0&\leq& \E\left(Z^{i}-\sum_{j<i}c_{ij} W^{j}\right)^2
=\left(1-2\sum_{j<i}c_{ij}^2\right)+\E\left(\sum_{j<i} c_{ij}W^{j}\right)^2=1-\sum_{j<i}c_{ij}^2,
\end{eqnarray*}
where the last equality follows from part 3.~of the induction hypothesis. This proves 2.
Moreover, $\E\left(Z^{i} W^{i}\right)=c_{ij}$ for $j=i$ follows now in the same way as above for $j<i$. This completes the proof of 1.

Next we show 3. 
If $c_{ii}=0$, the claim is trivial because $U^{i}$ is independent of everything else. If $c_{ii}>0$,  then for all $j<i$ by definition and the induction hypothesis,
\begin{eqnarray*}
\E\left( W^{i} W^{j}\right)=\frac{1}{c_{ii}}\E\left(\left(Z^{i}-\sum_{l<i}c_{il} W^{l}\right) W^{j}\right)
=\frac{1}{c_{ii}}\left[\E\left(Z^{i} W^{j}\right)-c_{ij}\right]\stackrel{1.}{=}0
\end{eqnarray*}
as well as
\begin{eqnarray*}
\E\left( W^{i}\right)^2=\frac{1}{c_{ii}^2}\cdot\E\left(Z^{i}-\sum_{j<i}c_{ij} W^{j}\right)^2
=\frac{1}{c_{ii}^2}\left(1-\sum_{j<i}c_{ij}^2\right)
=1.
\end{eqnarray*} 
Thus, 3.~is proven. It remains to show 4. If $c_{ii}\neq0$, 4.~is trivial. Hence, suppose that $c_{ii}=0$. Then,
\begin{eqnarray*}
\E\left(Z^{i}-\sum_{j<i}c_{ij}W^{i}\right)^2=\left(1-\sum_{j<i}c_{ij}^2\right)=c_{ii}^2=0,
\end{eqnarray*}
which shows that 4.~is also true in this case.
\end{proof}

Next we define for all $i \in \N$, $j\leq i$ numbers $\alpha_{ij}$ iteratively as follows:
\begin{align}
&\text{For }i=1,2,3,4,\dots:\notag\\
&\qquad \alpha_{ii}:=\begin{cases}\frac{1}{c_{ii}}&:\ c_{ii}>0\\1&:\ c_{ii}=0\end{cases}.\notag\\
&\qquad\text{For }j=i-1,i-2,\dots,1:\notag\\
&\qquad\qquad\text{If }c_{jj}=0,\notag\\
&\qquad\qquad\quad\text{then } \alpha_{ij}:=0.\notag\\
&\qquad\qquad\text{Else }\notag\\
&\qquad\qquad\quad \alpha_{ij}:=-\frac{1}{c_{jj}}\left(\sum_{j<l\leq i}\alpha_{il} c_{lj}\right).\label{asum}\\
&\qquad\text{Next $j$}.\notag\\
&\text{Next }i.\notag
\end{align}

Note that $\left(\alpha_{ij}\right)_{i\in\N,\ j\leq i}$ can be regarded as the "inverse" of  $\left(c_{ij}\right)_{i\in\N,\ j\leq i}$ in the following sense: for fixed $i,j\in\N$ with $j\leq i$ one has
\[\sum_{j\leq l\leq i}\alpha_{il}c_{lj} \stackrel{(\ref{asum})}{=}\1_{\left\{c_{jj}=0\right\}}\sum_{j\leq l\leq i}\alpha_{il}c_{lj}+\1_{\left\{c_{jj}^{(n)}>0\right\}}\delta_{ij}=\1_{\left\{c_{ii}>0\right\}}\delta_{ij},\]
where the last equality follows from the fact that $c_{lj}=0$ for all $l>j$ if $c_{jj}=0$. Hence, if $c_{ii}>0$, then
\begin{equation}\label{WZ}
\begin{split}
\sum_{j\leq i}\alpha_{ij}\ Z^{j}=\sum_{j\leq i}\alpha_{ij}\sum_{l\leq j}c_{jl} W^{l}
=\sum_{l\leq i}W^{l}\sum_{l\leq j\leq i}\alpha_{ij}c_{jl}=W^{i}.
\end{split}
\end{equation}

\section{Integration with respect to $Y^{(n)}$ and $Y$}\label{appendix-integration}

In this appendix we introduce the stochastic integrals with respect to $Y^{(n)}$ and $Y$. The concept of integration follows \cite{KP2}. 
We recall the definition of the random variables $\delta W^{(n),i}_k$ in (\ref{deltaW}) and put  ($i\in\N,\ t\in [0,T]$),
\[
	W^{(n),i}(t):=W^{(n)}(f_i,t)=\sum_{k=1}^{\lfloor t/\Delta t^{(n)}\rfloor}\delta W^{(n),i}_k\qquad\text{and}\qquad W^i:=W(f_i,\cdot)
\]
where $W$ is a cylindrical Brownian motion. Thus, the random variables 
the $W^i, \ i\in \N,$ are independent Brownian motions and each $Y^{(n)}=\left(Y^{(n)}_t\right)_{t\in[0,T]}$ is adapted to the filtration $\left(\hat{\F}^{(n)}_t\right)_{t\in[0,T]}$ defined via
\[
	\hat{\F}^{(n)}_t:=\F^{(n)}_k,\quad t_k^{(n)}\leq t<t_{k+1}^{(n)}.
\]
As integrands for $Y^{(n)}$ we consider c\`adl\`ag, $\left(\hat{\F}^{(n)}_t\right)_{t\in[0,T]}$-adapted processes which take their values in the space
\[\hat{E}:=\R\times L^2(\R_+;\R)\times\R\times L^2(\R_+;\R)\times L^2\left(\R^2_+;\R\right)\times L^2\left(\R_+;\R\right),\]
endowed with the norm
\[\left\Vert(a_1,a_2,a_3,a_4,a_5,a_6)\right\Vert_{\hat{E}}:=|a_1|+\left\Vert a_2\right\Vert_{L^2\left(\R_+\right)}+|a_3|+\left\Vert a_4\right\Vert_{L^2\left(\R_+\right)}+\left\Vert a_5\right\Vert_{L^2\left(\R^2_+\right)}+\left\Vert a_6\right\Vert_{L^2\left(\R_+\right)}.\]
We define $\mathcal{S}^{(n)}_{\hat{E}}$ as the set of processes $a^{(n)}: \Omega \times [0,T] \times \mathbb{R} \times \mathbb{R} \ra \hat{E}$ that are  of the form 
\begin{equation}\label{integrand}
\begin{split}
 & a^{(n)}(t;x,y) :=   \\ 
& \left(a^{1,(n)}(t),\sum_ja^{2,(n)}_j(t)f_j(y),a^{3,(n)}(t),\sum_ia^{4,(n)}_i(t)f_i(x),\sum_{ij}a^{5,(n)}_{ij}(t)f_i(x)f_j(y),\sum_ia^{6,(n)}_i(t)f_i(x)\right)
\end{split}
\end{equation}
for c\`adl\`ag and $\left(\hat{\F}^{(n)}_t\right)$-adapted processes $a^{1,(n)},a^{2,(n)}_j,a^{3,(n)},a^{4,(n)}_i,a^{5,(n)}_{ij},a^{6,(n)}_i$, $i,j\in\N$, of which all but finitely many are zero. For $a^{(n)}\in\mathcal{S}^{(n)}_{\hat{E}}$ with the representation as above, the integral with respect to $Y^{(n)}$ is defined as
\begin{equation*}\label{integral}
\begin{split}
& \int_0^t a^{(n)}(u-)dY^{(n)}(u):=\\
&\quad\left(\int_0^ta^{1,(n)}(u-)dZ^{(n)}(u)+\sum_j\int_0^ta^{2,(n)}_j(u-)dW^{(n),j}(u)+\sum_{k=1}^{\lfloor t/\Delta t^{(n)}\rfloor}a^{3,(n)}\left(t_k^{(n)}-\right)\Delta t^{(n)}, \right.\qquad\qquad\qquad\qquad\\ 
&\left.\qquad\quad\sum_{i}f_i\int_0^ta^{4,(n)}_{i}(u-)dZ^{(n)}(u)+\sum_{ij}f_i\int_0^ta^{5,(n)}_{ij}(u-)dW^{(n),j}(u)+\sum_if_i\sum_{k=1}^{\lfloor t/\Delta t^{(n)}\rfloor}a^{6,(n)}_i\left(t_k^{(n)}-\right)\Delta t^{(n)}\right).
\end{split}
\end{equation*}

\begin{thm}\label{ut}
Suppose that Assumptions \ref{prob1}, \ref{scaling}, \ref{density}, \ref{gi}, \ref{hi}, and \ref{U} hold. Then the sequence $Y^{(n)}$ is uniformly tight, i.e.
\[\mathcal{H}_t:=\bigcup_n\left\{\left\Vert \int_0^ta^{(n)}(u-)dY^{(n)}(u)\right\Vert_E:\ a^{(n)}\in \mathcal{S}^{(n)}_{\hat{E}},\ \sup_{u\leq t}\left\Vert a^{(n)}(u)\right\Vert_{\hat{E}}\leq1\ \text{a.s.}\right\}\]
is stochastically bounded for all $t\in[0,T]$.
\end{thm}

\begin{proof}
It is sufficient to show that for any $t\in[0,T]$ there exists a constant $C(t)$ such that for all $n\in\N$ and $a^{(n)}\in\mathcal{S}^{(n)}_{\hat{E}}$ with $\sup_{u\leq t}\left\Vert a^{(n)}(u)\right\Vert_{\hat{E}}\leq 1$,
\[\E\left\Vert \int_0^ta^{(n)}(u-)dY^{(n)}(u)\right\Vert_E\leq C(t).\]
Let $a^{(n)}\in\mathcal{S}^{(n)}_{\hat{E}}$ satisfy $\sup_{u\leq t}\left\Vert a^{(n)}(u)\right\Vert_{\hat{E}}\leq 1$. Thus for all $u\leq t$, 
{\small
\[\max\left\{\left|a^{1,(n)}(u)\right|,\sum_j\left(a^{2,(n)}_j(u)\right)^2,\left|a^{3,(n)}(u)\right|,\sum_i\left(a^{4,(n)}_i(u)\right)^2,\sum_{ij}\left(a^{5,(n)}_{ij}(u)\right)^2,\sum_i\left(a^{6,(n)}_i(u)\right)^2\right\}\leq 1\quad \text{a.s.}
\]}
This implies that
\[\E\left|\sum_{k=1}^{\lfloor t/\Delta t^{(n)}\rfloor}a^{3,(n)}\left(t^{(n)}_k-\right)\Delta t^{(n)}\right|\leq\Delta t^{(n)}\sum_{k=1}^{\lfloor t/\Delta t^{(n)}\rfloor}\E\left|a^{3,(n)}\left(t^{(n)}_k-\right)\right|\leq t\]
and, since only finitely many of the $a_{i}^{6,(n)}$ are assumed to be unequal zero,
\[\E\left\Vert \sum_{i}f_i\sum_{k=1}^{\lfloor t/\Delta t^{(n)}\rfloor}a^{6,(n)}_{i}\left(t^{(n)}_k-\right)\Delta t^{(n)}\right\Vert_{L^2}\leq\sum_{k=1}^{\lfloor t/\Delta t^{(n)}\rfloor}\Delta t^{(n)}\E\left(\sum_{i} \left(a^{6,(n)}_{i}\left(t_k^{(n)}-\right)\right)^2\right)^{1/2}\leq t.\]
For the other four terms recall that $a^{(n)}$ is $\left(\hat{\F}^{(n)}_t\right)$-adapted. Thus $a^{(n)}\left(t_k^{(n)}-\right)\in{\F}^{(n)}_{k-1}$ for $k=1,\dots, T_n$. So,
\begin{eqnarray*}
	\E\left(\int_0^ta^{1,(n)}(s-)dZ^{(n)}(s)\right)^2 &=& \E\left(\sum_{k=1}^{\lfloor t/\Delta t^{(n)}\rfloor}a^{1,(n)}\left(t^{(n)}_k-\right)\delta Z_k^{(n)}\right)^2 \\ &\leq& \Delta t^{(n)}\sum_{k=1}^{\lfloor t/\Delta t^{(n)}\rfloor}\E\left(a^{1,(n)}\left(t^{(n)}_k-\right)\right)^2 \leq t
\end{eqnarray*}
and, since only finitely many of the $a_{ij}^{4,(n)}$ are assumed to be unequal zero,
\begin{eqnarray*}
\E\left\Vert\sum_if_i\int_0^ta^{4,(n)}_i(s-)dZ^{(n)}(s)\right\Vert_{L^2}^2&=&\sum_i\E\left(\sum_{k=1}^{\lfloor t/\Delta t^{(n)}\rfloor}a^{4,(n)}_i\left(t^{(n)}_k-\right)\delta Z_k^{(n)}\right)^2\\
&\leq& \Delta t^{(n)}\sum_{k=1}^{\lfloor t/\Delta t^{(n)}\rfloor}\sum_i\E\left(a^{4,(n)}_i\left(t^{(n)}_k-\right)\right)^2 \leq t.
\end{eqnarray*}
Similarly, since only finitely many of the $a_{j}^{2,(n)}$ and $a_{ij}^{5,(n)}$ are assumed to be unequal zero,
\begin{eqnarray*}
\E\left|\sum_{j}\int_0^ta^{2,(n)}_{j}(u-)dW^{(n),j}\left(u\right)\right|
&\leq& \sum_{k=1}^{\lfloor t/\Delta t^{(n)}\rfloor} \E\left|\sum_ja^{2,(n)}_{j}\left(t_k^{(n)}-\right)\delta W^{(n),j}_k\right|\\
&\leq&\sum_{k=1}^{\lfloor t/\Delta t^{(n)}\rfloor}\left(\E\left[\sum_{j}a^{2,(n)}_{j}\left(t_k^{(n)}-\right)\delta W^{(n),j}_k\right]^2\right)^{1/2} \\
&\leq& \sum_{k=1}^{\lfloor t/\Delta t^{(n)}\rfloor}\Delta t^{(n)}\E\sum_j\left(a^{2,(n)}_{j}\left(t_k^{(n)}-\right)\right)^2 \leq t
\end{eqnarray*}
and 
\begin{eqnarray*}
\E\left\Vert \sum_{ij}f_i\int_0^ta^{5,(n)}_{ij}(u-)dW^{(n),j}\left(u\right)\right\Vert_{L^2}^2 &=&
\sum_{i} \E\left(\sum_j\int_0^ta^{5,(n)}_{ij}(u-)dW^{(n),j}(u)\right)^2\\
&=& \sum_{i}\sum_{k=1}^{\lfloor t/\Delta t^{(n)}\rfloor} \E\left(\sum_ja^{5,(n)}_{ij}\left(t_k^{(n)}-\right)\delta W^{(n),j}_k\right)^2\\
&=& \sum_{ij} \Delta t^{(n)}\sum_{k=1}^{\lfloor t/\Delta t^{(n)}\rfloor}\E\left(a^{5,(n)}_{ij}\left(t_k^{(n)}-\right)\right)^2 \leq t.
\end{eqnarray*}
\end{proof}

The preceding theorem implies that $Y^{(n)}=\left(Y^{(n)}_t\right)_{t\in[0,T]}$ is a standard $(E,\hat{E})$-semimartingale in the sense of \cite{KP2}. Therefore, the definition of the stochastic integral $\int a^{(n)}_-dY^{(n)}$ extends to all c\`adl\`ag, adapted, uniformly bounded, $\hat{E}$-valued processes $a^{(n)}$, where the resulting infinite sums can be shown to exist as limits in probability. 

Similarly, if $(\F_t)_{t\in[0,T]}$ is any filtration to which $Y=(Y_t)_{t\in[0,T]}$ is adapted, we will denote by $\mathcal{S}_{\hat{E}}$ the set of $\hat{E}$-valued processes $a$ of the form (\ref{integrand}), for which all $a^1,a^2_j,a^3,a^4_i,a^5_{ij},a^6_i,\ i,j\in\N$, are c\`adl\`ag, $\left(\F_t\right)$-adapted processes, of which all but finitely many are zero. The integral of $a\in\mathcal{S}_{\hat{E}}$ with respect to $Y$ is then defined by, 
\begin{eqnarray*}
\int_0^t a(u-)dY(u)&:=&\left(\int_0^ta^1(u-)dZ(u)+\sum_{j}\int_0^ta^2(u-)dW^j(u)+\int_0^ta^3(u)du,\right.\\
&&\left.\qquad\qquad\sum_i\int_0^ta^4_i(u-)dZ(u)+\sum_{ij}f_i\int_0^ta^3_{ij}(u-)dW^j(u)+\sum_if_i\int_0^ta^6_i(u)du\right).
\end{eqnarray*}
Analogously as above, one can show that $Y$ is also an $(E,\hat{E})$-semimartingale and thus we can again extend the definition of the integral $\int a_-dY$ to all c\`adl\`ag, $(\F_t)$-adapted $\hat{E}_{loc}$-valued processes $a$.

In view of (\ref{SDE_BV}) we only need to consider integrands of the form (\ref{integrand}), for which $a^{2,(n)}_j\equiv0$ and $a^{4,(n)}_i\equiv0$ for all $i,j\in\N$. Moreover, we will further extend the definition of the integral $\int a^{(n)}(u-)dY^{(n)}(u)$ allowing as integrands all c\`adl\`ag, adapted processes $a^{(n)}$ which take their values in the set
\[\hat{E}_{loc}:=\R\times\{0\}\times\R\times \{0\}\times L_{loc,diag}^2(\R_+^2;\R)\times L_{loc}^2(\R_+;\R),\]
where 
\[L_{loc,diag}^2(\R_+;\R):=\left\{\left.h(x,y)=\sum_i\sum_{j\leq i}h_{ij}f_i(x)f_j(y)\ \right|\ \sum_{i\in\mathcal{I}_m}\sum_{j\leq i}h_{ij}^2<\infty \ \forall\ m\in\N\right\}.\]
The definition of the integral will be extended as follows: for any $a^{(n)}\in\hat{E}_{loc}$, the process 
\[\int_0^t a^{(n)}(u-)dY^{(n)}(u),\ t\in[0,T],\]
is defined as the unique c\`adl\`ag $E_{loc}$-valued process $X=(X^1,X^2)$ (up to indistinguishability) such that for all $m\in\N$ and rational $t\in[0,T]$,
\begin{equation}\label{integralm}
\left(X_t^1,\1_{[0,m]}X^2_t\right)=\int_0^ta^{(n),m}(u-)dY^{(n)}(u),
\end{equation}
where the $\hat{E}$-valued processes $a^{(n),m},\ m\in\N,$ are defined as the projections of $a^{(n)}$ on the subspace $\{f_i:\ i\in\mathcal{I}_m\}$ with $\mathcal{I}_m$ defined in (\ref{im}):
\[a^{(n),m}(t;x,y):=\left(a^{1,(n)}(t),\ 0,\ a^{3,(n)}(t),\ 0,\ \sum_{i\in\mathcal{I}_m}\sum_{j\leq i}a_{ij}^{5,(n)}(t)f_i(x)f_j(y),\ \sum_{i\in\mathcal{I}_m}a_i^{6,(n)}(t)f_i(x)\right).\]

{\small

\bibliographystyle{abbrv}

\bibliography{LOBLit}
}

\end{document}